\documentclass[11pt,twoside]{article}

\usepackage{mathrsfs,amsfonts,amsmath,amssymb}
\usepackage{latexsym}
\newtheorem{satz}{Theorem}
\newtheorem{proposition}[satz]{Proposition}
\newtheorem{theorem}[satz]{Theorem}
\newtheorem{lemma}[satz]{Lemma}

\newtheorem{corollary}[satz]{Corollary}
\newtheorem{remark}[satz]{Remark}
\newtheorem{example}[satz]{Example}

\def\eps{\varepsilon}
\def\_phi{\varphi}

\def\a{\alpha}
\def\d{\delta}
\def\la{\lambda}

\def\F{{\mathbb F}}

\def\o{\omega}
\def\ov{\overline}

\def\C{{\mathbb C}}
\def\R{{\mathbb R}}
\def\E{\mathsf {E}}
\def\T{{\mathbb T}}

\def\Z_N{{\mathbb Z}_N}
\def\Z{{\mathbb Z}}

\def\f{{\mathbb F}}
\def\Gr{{\mathbf G}}

\def\D{{\mathbb D}}

\def\supp{{\rm supp\,}}
\def\tr{{\rm tr\,}}

\def\FF{\widehat}

\def\D{\Delta}

\def\T{\mathsf {T}}

\def\SL{{\rm SL\,}}

\author{Shkredov I.D.}
\title{Modular hyperbolas and bilinear forms of Kloosterman sums
	\footnote{This work is supported by the Russian Science Foundation under grant 19--11--00001.}
}
\date{}
\begin{document}
	\maketitle

	\begin{center}
		Annotation.
	\end{center}

	{\it \small
		In this paper we study incidences for hyperbolas in $\F_p$ and show how linear sum--product  methods work for such curves. 
		As an application we give a purely combinatorial proof of 
		%obtain 
		a nontrivial upper bound for 
		%symmetric 
		bilinear forms of Kloosterman sums.
		% in the range from $(\log p)^C$ to $p^{1/21}$. 
	}
	\\
	%\\
	%\\

	\section{Introduction}
	\label{sec:introduction}

	Let $p$ be an odd prime number, and $\F_p$ be the finite field.
	Given two sets $A,B\subset \F_p$, define the  \textit{sumset}, the \textit{product set} and the \textit{quotient set} of $A$ and $B$ as 
	$$A+B:=\{a+b ~:~ a\in{A},\,b\in{B}\}\,,$$
	$$AB:=\{ab ~:~ a\in{A},\,b\in{B}\}\,,$$
	and 
	$$A/B:=\{a/b ~:~ a\in{A},\,b\in{B},\,b\neq0\}\,,$$
	correspondingly.
	This paper is devoted to the so--called {\it sum--product phenomenon},
	%over 
	%in 
	%The sum--product phenomenon asserts 
	which says that either the sumset or the product set of a set must be large up to some natural algebraic constrains.  
	One of 
	%the most difficult and 
	the strongest form of this principle is the Erd\H{o}s--Szemer\'{e}di  conjecture \cite{ES},
	which says that for any sufficiently large set $A$ of reals and an arbitrary $\epsilon>0$ one has
	\begin{equation*}\label{f:sum-product_ES}
	\max{\{|A+A|,|AA|\}}\gg{|A|^{2-\epsilon}} \,.
	\end{equation*}
	The best up to date results in the direction 
	can be found in \cite{Shakan}
	%\cite{soly}, \cite{KS1}, \cite{KS2}, \cite{RSS} 
	and in 
	%\cite{AMRS}, \cite{RRS}  
	\cite{RSS} 
	for  $\R$ and $\F_p$, respectively. 
	Basically, in this 
	paper we restrict ourselves to the case of the finite fields only.

	It was Elekes \cite{Elekes2} who realised that the sum--product phenomenon is connected with Incidence Geometry.
	%The main subject of the Incidence Geometry is a collection of 
	Incidence Geometry deals with the incidences among basic geometrical objects such as points, lines, curves, surfaces and so on. 
	After Elekes various results on incidences of different types in $\R$ were obtained by many authors (see, e.g., \cite{TV}). 
	%many authors obtained
	Nevertheless, in $\F_p$ only {\it linear} incidences, i.e., incidences between linear objects as points/planes, points/lines, lines/lines were obtained 
	see, e.g., \cite{Rudnev_pp}, \cite{SdZ}, \cite{Zeeuw}. 
	A remarkable exception is the case of so--called $\SL_2 (\F_p)$--hyperbolas and this exception 
	%an approach which 
	was suggested by Bourgain \cite{B_hyp} who gives, in particular, the first nontrivial upper bound for cardinality of the following set
	\begin{equation}\label{f:hyp}
	\{ (a+b) (c+d) = \lambda ~:~ a\in A,\, b\in B,\, c\in C,\, d\in D \} 
	\end{equation}
	for any $\la \neq 0$ and arbitrary sets $A,B,C,D \subseteq \F_p$. 
	The importance of hyperbolas in Additive Combinatorics and Number Theory was discussed in \cite{Shp_hyp}. 
	Bourgain's approach was connected with the group actions (the importance of the group actions in Additive Combinatorics was realized by Elekes as well, see \cite{Elekes1}, \cite{Elekes3}) and it was based on Helgott's result on growth  in  $\SL_2 (\F_p)$, see  \cite{HH}, \cite{HaH} and on some additional considerations \cite{BG}.
	% as well as . 

	In this paper we 
	% continue 
	obtain a series of new upper bounds for cardinality of the set from \eqref{f:hyp}.
	Here are two of our results (other results can be found in Sections \ref{sec:hyp_first}, \ref{sec:hyp_second}, see, e.g., Theorem \ref{t:B_progr_new} below).  
	
	\begin{theorem}
		Let $A,B,C,D\subseteq \F_p$ be sets.
		Then for any $\la \neq 0$, one has 
		\[
		|\{ (a+b) (c+d) = \lambda ~:~ a\in A,\, b\in B,\, c\in C,\, d\in D \}| - \frac{|A||B||C||D|}{p} 
		\lesssim 
		\]
		\begin{equation}\label{f:hyp_incidences_intr}
		\lesssim |A|^{1/4} |B||C| |D|^{1/2} + |A|^{3/4} (|B||C|)^{41/48} |D|^{1/2}   \,.
		\end{equation}
		\label{t:hyp_incidences_intr}
	\end{theorem}

	The Theorem above allows to obtain a uniform upper bound for size of hyperbola with elements from a set with the small sumset.

	\begin{corollary}
		Let $A\subseteq \F_p$ be a set.
		Suppose that $|A+A| \ll |A|$ and $|A| \ll p^{13/23}$. 
		Then for any $\la \neq 0$, one has 
		\begin{equation}\label{f:r_AA'_intr}
		|\{ a_1 a_2 = \la ~:~ a_1,a_2 \in A \}| \lesssim  
		%\frac{|A|^2}{p} + 
		|A|^{149/156} \,.
		\end{equation}
		\label{c:r_AA_intr}
	\end{corollary}

	Another rather unusual result (for example, the proof uses the fact that the group $\SL_2 (\Z)$ contains free subgroups) on incidences \eqref{f:hyp} is the following (an analogue of this statement in $\F_p$ is our  Theorem \ref{t:B_progr_new} from Section \ref{sec:hyp_second}). 
	
	\begin{theorem}
		Let $A,D \subset \R$, $B\subset \Z$  be sets, and $\la \neq 0$ be any number.
		Then 
		%	\begin{equation}\label{f:A_abs_Re_Z}
		\[
		|\{ (a+b) (c+d) = \lambda ~:~ a\in A,\, d\in D,\, b,c \in B \}|
		%	\ll_{|\la|}
		\] 
		\begin{equation}\label{f:A_abs_Re_Z_intr}
		\ll_{|\la|} 
		\sqrt{|A||D|} |B|^2 \cdot \max\{ |D|^{-1/2}, 
		%(\E^{+} (B) |B|^{-4} )^{1/5} \}  \,.
		|B|^{-1/4} \}  \,.
		\end{equation}
		%	where $c>0$ is an absolute constant. 
		\label{t:A_abs_Re_Z_intr} 
	\end{theorem}

	Rather mysterious part of Helfgott's proof of $\SL_2 (\F_p)$--growth result was that the sum--product phenomenon in $\F_p$, which deals exclusively with linear objects as points/lines, points/planes and so on gives absolutely nontrivial results for completely different curves, namely, for hyperbolas.
	An explanation in a particular but a transparent case is given in our Lemma \ref{l:T_3(G)},  where we estimate a certain energy of a subset of matrices from $\SL_2 (\F_p)$ via purely linear  sum--product quantity. Now it remains to notice that energies of subsets of acting groups are naturally 
	%connected 
	related 
	with the incidences, see, e.g., \cite{NG_S}, \cite{Brendan_rich}, \cite{RS_SL2}.

	It turns out that incidences between hyperbolas and points are connected with bilinear forms of Kloosterman sums, see 
	\cite{BFKMM}, \cite{FKM}, \cite{KS}--\cite{MS_Kloosterman}, 
	%\cite{KMS_gen},
	%, \cite{KMS}    
	\cite{Shp_hyp}--\cite{SZ_Kloosterman} and other papers.
	We obtain the following result in this  direction (see Theorems \ref{t:Kloosterman}, \ref{t:Kloosterman_NM} from Section \ref{sec:Kloosterman}), 
	which we  formulate here in a particular case (the main advantage 
	%strength 
	of our method is that it allows to consider rather general sets and weights). 
	Recall that the Kloosterman sum in a  finite field $\F$  is  
	\[
	K(n,m) = \sum_{x \in \F \setminus \{0\}} e( nx + mx^{-1}) \,.
	\]
	We are interested in bilinear forms of Kloosterman sums \cite{KMS}--\cite{KMS_gen}, that is, the sums of the form 
	\[
	S(\a,\beta) = \sum_{n,m} \a(n) \beta (m) K(n,m) \,,
	\] 
	where $\a : \F \to \C$, $\beta : \F \to \C$ are rather arbitrary  functions.

	\begin{theorem}
		Let $\a,\beta  : \F_p \to \C$ be functions with supports on $\{1,\dots, N\} +t_1$ and $\{1,\dots, M\} +t_2$, respectively,  
		and 
		$N$ or $M$ is at most $p^{1-c}$, $c>0$.
		% and $A=\supp \a$, $B=\supp \beta$.
		Then
		\begin{equation}\label{f:d-est_intr}
		S(\a,\beta ) \lesssim \| \a \|_2 \| \beta \|_2  p^{1-\d}  \,,
		\end{equation}
		where 
		%$\d>0$ 
		$\d (c) >0$ 
		is a positive constant. 
		Besides, if  
		%$\beta_2 = \{1,\dots, M\} + t_2$ and 
		$M^2 <pN$, then
		%$\supp \beta \subseteq \{1,\dots, M\}$, then  
		\begin{equation}\label{f:Kloosterman_NM_1_intr}
		S(\{1,\dots, N\}+t_1,\beta) \lesssim \| \beta \|_2 \left(N^{3/7} M^{1/7} p^{13/14} +  N^{3/4} p^{3/4} + N^{1/4} p^{13/12} \right) \,.
		\end{equation}
		%	Finally, if $N=M \le p^{1/21}$ and $t_1 =t_2= 0$, then 
		%	\begin{equation}\label{f:Kloosterman_NM_3_intr}
		%	S(\a,\beta) \lesssim \sqrt{p} \| \a\|_2 \| \beta \|_2 N^{63/64}  \,.
		%	\end{equation}	
		\label{t:Kloosterman_intr}
	\end{theorem}

	%Bound \eqref{f:Kloosterman_NM_3_intr} says that in the case of zero shifts and the same arithmetic progression, we have a nontrivial saving for sums $S(\a,\beta)$ for very wide range of parameters, namely, $(\log p)^C \ll N \le p^{1/21}$, where $C>1$ is an absolute constant (we assume for simplicity that  $\a$ and $\beta$ are bounded).  
	%Here $\| f\|_W$ is normalized $L^1$--norm of Fourier transform, see Section \ref{sec:definitions}. 
	It is easy to check that the last result is better than \cite[Theorem 7]{Shp_Sato}, as well as \cite[Theorem 1.17(2)]{FKM} 
	but worse than the current world record from \cite{KMSII} in the case $N=M$.
	%Previously, such savings were known \cite{KMSII} just for $N\ge p^{3/8+\eps}$, $\eps>0$  and for $N\ge p^{1/3+\eps}$, $\eps>0$ in the case $\a=\{1,\dots, N\}$ 
	Of course the advantage of our results is that they hold in  very general situation.
	Also, the method of the proof is not analytical but combinatorial  one and hence does not require deep tools from Algebraic Geometry as in \cite{KMSII}. 
	% (see Theorem \ref{t:Kloosterman} below).  
	%For large powers of $p$ a similar saving as in Theorem \ref{t:Kloosterman_intr} was obtained in \cite{KST}.

	%Of course bound \eqref{f:Kloosterman_NM_3_intr} of Theorem \ref{t:Kloosterman_intr} implies a number--theoretical consequence.

	%\begin{corollary}
	%	Let $p$ be a prime and $(\log p)^C \ll X \ll p^{1/21}$ for a certain absolute constant $C>1$. 
	%	Then 
	%\[
	%	\sum_{n=1}^X \tau(n) K(n,1) \ll X^{1-c} \,,
	%\] 
	%	where $c<1/64$.
	%\end{corollary}

	In 
	%the 
	our 
	paper we develop the ideas from \cite{NG_S}, where growth in $\SL_2 (\F_p)$ was applied to the Zaremba conjecture about continued fractions. 
	Before this paper various analytical tools (as Kloosterman sums)  
	%applied 
	were used 
	in 
	%this 
	%the discussed 
	the aforementioned 
	area, see, e.g., \cite{Mosh_Kloosterman}.
	After 
	%this paper 
	\cite{NG_S} 
	it is not surprising that sometimes  combinatorial methods give results of comparable quality 
	%and better 
	to the ones, which were obtained via deep analytical techniques.

	%We conclude with few comments regarding the notation used in this paper. 
	%For a positive integer $n,$ we set $[n]=\{1,\ldots,n\}.$ 
	%All logarithms are base $2.$ Signs $\ll$ and $\gg$ are the usual Vinogradov's symbols. 
	%Finally, with a slight abuse of notation we use the same letter to denote a set $S\subseteq \Gr$ 
	%and its characteristic function $S:\Gr\rightarrow \{0,1\}.$

	All logarithms are to base $2.$ The signs $\ll$ and $\gg$ are the usual Vinogradov symbols.
	For a positive integer $n,$ we set $[n]=\{1,\ldots,n\}.$
	Having a set $A$,  we will write $a \lesssim b$ or $b \gtrsim a$ if $a = O(b \cdot \log^c |A|)$, $c>0$.

	The author is grateful to Brendan Murphy, Nikolay Moshchevitin, Dmitrii Frolenkov and Maxim Korolev for very useful discussions and fruitful explanations.

%\section{Notation and preliminaries}
\section{Notation}
\label{sec:definitions}

In this paper $\F$ is a field, and $p$ is an odd prime number, 
%further 
$\F_p = \Z/p\Z$ and $\F_p^* = \F_p \setminus \{0\}$. 
%
%
%%If $\gamma \in \f_p^*$ then $f^\gamma (x) := f(\gamma x)$.
We denote the Fourier transform of a function  $f : \F_p \to \mathbb{C}$ by~$\FF{f},$ namely, 
\begin{equation}\label{F:Fourier}
\FF{f}(\xi) =  \sum_{x \in \F_p} f(x) e( -\xi \cdot x) \,,
\end{equation}
where $e(x) = e^{2\pi i x/p}$. 
We rely on the following basic identities. 
The first one is called the Plancherel formula and its particular case $f=g$ is called the Parseval identity 
\begin{equation}\label{F_Par}
\sum_{x\in \F_p} f(x) \ov{g (x)}
=
\frac{1}{p} \sum_{\xi \in \F_p} \widehat{f} (\xi) \ov{\widehat{g} (\xi)} \,.
\end{equation}
Another  particular case of (\ref{F_Par}) is 
\begin{equation}\label{svertka}
\sum_{y\in \F_p} |(f*g) (x)|^2 
=
\sum_{y\in \F_p} \Big|\sum_{x\in \F_p} f(x) g(y-x) \Big|^2
= \frac{1}{p} \sum_{\xi \in \F_p} \big|\widehat{f} (\xi)\big|^2 \big|\widehat{g} (\xi)\big|^2 \,.
\end{equation}
and the identity 
%%the inversion formula
\begin{equation}\label{f:inverse}
f(x) = \frac{1}{p} \sum_{\xi \in \F_p} \FF{f}(\xi) e(\xi \cdot x) 
\,.
\end{equation}
is called the inversion formula.
The (normalized) Wiener norm of $f(x)$ is defined as 
\begin{equation}\label{def:Wiener}
\| \FF{f} \|_{L^1} = \| f\|_W := \frac{1}{p} \sum_{\xi \in \F_p} |\FF{f}(\xi)| \,.
\end{equation}
Clearly, by the Parseval identity \eqref{F_Par}, the inverse formula \eqref{f:inverse} and the Cauchy--Schwarz inequality, we have 
\begin{equation}\label{f:Wiener}
\|f \|_\infty \le \| f\|_W \le \|f\|_2 \quad \quad \mbox{and} \quad \quad  \| f\|_1 \le p\| f\|_W \,.
\end{equation}
It is well--known that equipped with the Wiener norm the set of functions on the group forms an algebra relatively pointwise multiplication.   
In this paper we use the same letter to denote a set $A\subseteq \F$ and  its characteristic function $A: \F \to \{0,1 \}$.  
Also, we write $f_A (x)$ for the {\it balanced function} of a set $A\subseteq \F_p$, namely, $f_A (x) = A(x) - |A|/p$.
Let  $m \cdot A$ 
%for the set 
be the set 
$\{ ma ~:~ a\in A\}$.  
%Finally, having two functions $f(x), g(x)$ a new function $(f \otimes g) (x,y)$ is $(f \otimes g) (x,y) = f(x) g(y)$. 

%	We recall the notion of the spectrum $\Spec_\eps (A)$ of a set $A$.
%%	 and formulate the required result about the structure of $\Spec_\eps (A)$.
%	Let $A\subseteq \F_p$ be a set, and $\eps \in (0,1]$ be a real number.
%	Define
%	$$
%	\Spec_\eps (A) = \{ r \in \F_p ~:~ |\FF{A}(r)| \ge \eps |A| \} \,.
%	$$

Put
$\E^{+}(A,B)$ for the {\it common additive energy} of two sets $A,B \subseteq \f_p$
(see, e.g., \cite{TV}), that is, 
$$
\E^{+} (A,B) = |\{ (a_1,a_2,b_1,b_2) \in A\times A \times B \times B ~:~ a_1+b_1 = a_2+b_2 \}| \,.
$$
If $A=B$, then  we simply write $\E^{+} (A)$ instead of $\E^{+} (A,A)$
and the quantity $\E^{+} (A)$ is called the {\it additive energy} in this case. 
One can consider $\E^{+}(f)$ for any complex function $f$ as well.  
More generally, 
we deal with 
%consider 
a higher energy
\begin{equation}\label{def:T_k_ab}
\T^{+}_k (A) := |\{ (a_1,\dots,a_k,a'_1,\dots,a'_k) \in A^{2k} ~:~ a_1 + \dots + a_k = a'_1 + \dots + a'_k \}|
=
\frac{1}{p} \sum_{\xi} |\FF{A} (\xi)|^{2k}
\,.
\end{equation}
The last identity follows from  (\ref{svertka}). 
Another sort of higher energy is \cite{SS1} 
\[
\E^{+}_k (A) = |\{ (a_1,\dots,a_k,a'_1,\dots,a'_k) \in A^{2k} ~:~ a_1-a'_1 = \dots = a_k - a'_k \}| \,.
\]
Sometimes we  use representation function notations like $r_{AB} (x)$ or $r_{A+B} (x)$, which counts the number of ways $x \in \F_p$ can be expressed as a product $ab$ or a sum $a+b$ with $a\in A$, $b\in B$, respectively. 
%%	For example, $|A| = r_{A-A}(0)$ and  $\E^{+} (A) = r_{A+A-A-A}(0)=\sum_x r^2_{A+A} (x) = \sum_x r^2_{A-A} (x)$.  
%	Thus $r_{A+B} (x) = (A*B) (x)$, say.
%	Having $P\subseteq A-A$ we write $\sigma_P (A) := \sum_{x \in P} r_{A-A} (x)$. 
%Put $\sigma^{+} (A) = \sum_{x \in A} r_{A-A} (x)$.  
Further clearly
\begin{equation*}\label{f:energy_convolution}
\E^{+} (A,B) = \sum_x r_{A+B}^2 (x) = \sum_x r^2_{A-B} (x) = \sum_x r_{A-A} (x) r_{B-B} (x)
\end{equation*}
and by (\ref{svertka}),
%we have
%the following holds
\begin{equation}\label{f:energy_Fourier}
\E^{+}(A,B) = \frac{1}{p} \sum_{\xi} |\FF{A} (\xi)|^2 |\FF{B} (\xi)|^2 \,.
\end{equation}
Similarly, one can define $\E^\times (A,B)$, $\E^{\times} (A)$, $\E^{\times} (f)$  and so on.

\section{Preliminaries}
\label{sec:preliminaries}

We need in  a sum--product result from \cite[Theorem 32]{sh_as}, as well as 
%a result 
\cite[Theorem 35]{collinear}.

\begin{lemma}
	Let $A,B \subseteq \F_p$ be sets.
	Then
	\[
	\sum_x r^{2}_{(A-A)(B-B)} (x) - \frac{|A|^4 |B|^4}{p} \lesssim (|A| |B|)^{5/2} \E^{+} (A,B)^{1/2} \,.
	\]
	\label{l:D_2}
\end{lemma}

\begin{lemma}
	Suppose that $A$ is a subset of $\F_p$ such that $|A+A| = K|A|$ and $|A| \le p^{13/23} K^{25/92}$. 
	Then
	\[
	\E^\times (A) \lesssim K^{51/26} |A|^{32/13} \,.
	\] 
	\label{l:E_times_A}
\end{lemma}

In \cite[Theorem 2]{CSZ} it was obtained a very precise result on multiplicative energy of arithmetic progressions.

\begin{theorem}
	Let $A$ and $B$ be  arithmetic progressions with the difference equals one. Then 
	\[
	\E^\times (A,B) = \frac{|A|^2 |B|^2}{p} + O(|A||B| \log^2 p) \,.
	\]
	\label{t:E^*_progr}
\end{theorem}

%We need a simple result, see, e.g., \cite[Lemma 4.1]{SZ_Kloosterman}.

%\begin{lemma}
%	For any integers $1 \le X,Y < p$ and $\la \in \F^*_p$, the congruence 
%\[
%	xy \equiv \la \pmod \,, \quad \quad 1\le |x| \le X,\, 1 \le |y| \le Y 
%\]
%	has at most $(XY/p+1) p^{o(1)}$ solutions.
%\label{l:hyp_AP}
%\end{lemma}

%\begin{theorem}
%\end{theorem}

The 
%next 
last 
result about Fourier transform of arithmetic progressions is well--known. 

\begin{lemma}
	Let $P$ be an arithmetic progression. Then $\| P \|_W \ll \log p$ and for any $c>1$ the following holds 
	\[
	p^{-1} \sum_{\xi \in \F_p} |\FF{P} (\xi)|^c \ll |P|^{c-1} \,.  
	\]
	\label{l:P_norms}
\end{lemma}

\section{Some non--abelian results}
\label{sec:non-commutative}

%\bigskip
%$\hfill\Box$

We will formulate and  prove a series of results, which hold in general groups although, of course, our main applications concerns $\SL_2 (\F_p)$ and $\SL_2 (\Z)$.

Let $\Gr$ be a group and $A_1,\dots, A_{2k} \subseteq \Gr$ be sets.
For $k\ge 2$ put 
\begin{equation}\label{def:T_k}
\T_{k} (A_1,\dots, A_{2k}) = |\{ a_1 a^{-1}_2  \dots a_{k-1} a^{-1}_k = a_{k+1} a^{-1}_{k+2}  \dots a_{2k-1} a^{-1}_{2k} ~:~ a_j \in A_j \}| \,.
\end{equation}
More generally, one can define $\T_{k} (f_1,\dots, f_{2k})$ for any functions  $f_1,\dots,f_{2k} : \Gr \to \C$.
Basically, we are interested in the case of the characteristic functions $f_j$. 
If $k=2$, then  we write $\E$ instead of $\T_2$ as in Section \ref{sec:definitions}. 
For any $g\in \Gr$ one has 
\begin{equation}\label{f:T_k_g}
\T_{k} (g A_1,\dots, g A_{2k}) = \T_{k} (A_1,\dots, A_{2k}) 
~\mbox{ and }~  
\T_{k} (A_1 g,\dots, A_{2k} g) = \T_{k} (A_1,\dots, A_{2k}) \,.
\end{equation}
One of the reasons that we have defined $\T_{k} (A_1,\dots, A_{2k})$ as in \eqref{def:T_k} 
%(let us call coin the term "teethed"\, way, of defining $\T_k$)
%(we say that it is a {\it canonical} way of writing variables)
is that we have 
%invariance 
property \eqref{f:T_k_g}.
%Clearly, 
For example, 
if $\T_{k} (A_1,\dots, A_{2k})$ was defined as just the number of the solutions to the equation 
$
a_1 \dots a_k = a_{k+1} \dots a_{2k},
$
then formula \eqref{f:T_k_g} fails. 
Another reason is that for such defined $\T_k$ Lemmas \ref{l:T_2^k}, \ref{l:actions} take place.
Finally, in terms of eigenvalues of some operators, see the proof of Lemma \ref{l:actions} and Remark \ref{r:T_k_non-abelian}, we have for such $\T_k$ a full analogue with the abelian case, compare formula \eqref{def:T_k_ab} and formula \eqref{f:T_k_eigenvalues}.

Now if $A_1=\dots = A_{2k} = A$, then we write $\T_{k} (A)$ for  $\T_{k} (A,\dots, A)$ and similarly $\T_{k} (f)$ for $\T_{k} (f,\dots, f)$.
It is convenient to put $\T_1 (A) = |A|^2$. 
Also, denote $\T_{2k} (A,B) := \T_{2k} (A,B, \dots, A,B)$.  
%One can easily check that besides $\T_{k} (A) = \T_{k} (gA) = \T_{k} (Ag)$, $g\in \Gr$  one has $\T_{k} (A) = \T_{k} (A^{-1})$. 
Since $\T_{k} (A) = \T_{k} (gA) = \T_{k} (Ag)$ for any  $g\in \Gr$, it follows that in any matrix group 
(as $\SL_n$) an arbitrary permutation of rows or columns preserves  $\T_k$. 
Also, notice that  $\T_{k} (A) = \T_{k} (A^{-1})$. 
Further, $\T_{k} (A) \le |A|^{2(k-l)} \T_{l} (A)$, $l\le k$ because the operator, which fix any $l$ positions from the left side and from the right side in \eqref{def:T_k} is, clearly, symmetric and nonnegatively defined
(obviously, one has $\sum_{x,y\in M} \T(x,y) \le |M| \sum_{x\in M} \T(x,x)$ for any nonnegative operator $\T(x,y)$ defined on a set $M$). 
Using the Cauchy--Schwarz inequality, we have
\begin{equation}\label{f:r_infty}
\| r_{(AA^{-1})^{k}} \|_\infty ,\, \| r_{(A^{-1}A)^{k}} \|_\infty \le \T_{k} (A) \,.
\end{equation}
Further for $k\ge 1$  consider the higher energies \cite{SS1} 
\[
\E^R_k (A) = \sum_{x} r^{k}_{AA^{-1}} (x) = \sum_{x_1,\dots,x_{k-1}} |A \cap Ax_1 \cap \dots \cap A x_{k-1}|^2 
\]
and, similarly, $\E^L_k (A)$.

\bigskip

We need in a lemma about quantities $\T_k (A_1,\dots, A_{2k})$.
% which we prove for simplicity just in the case of dyadic indexes.

\begin{lemma}
	Let $f_1,\dots,f_{2k} : \Gr \to \C$ be functions.
	Then
	\begin{equation}\label{f:T_2^k}
	\T^{2k}_{k} (f_1,\dots, f_{2k}) \le \prod_{j=1}^{2k} \T_{k} (f_j) \,.
	\end{equation}
	In particular, for any $A,B,C,D \subseteq \Gr$ one has 
	\begin{equation}\label{f:T_2^k_E}
	\E(A,B,C,D)^4 \le \E(A) \E(B) \E(C) \E(D) \,.
	\end{equation}
	\label{l:T_2^k}
\end{lemma}
\begin{proof}
	For typographical reasons we will  assume sometimes that $f_j = A_j$ for some sets $A_j \subseteq \Gr$. 
	Clearly, by the Cauchy--Schwarz inequality for any $l$ 
	\[
	\T^2_{l} (A_1,\dots, A_{2l}) \le \T_{l} (A_1,\dots, A_{l}, A_1,\dots, A_{l}) \T_{l} (A_{l+1},\dots, A_{2l}, A_{l+1},\dots, A_{2l})    
	\]
	and thus it is enough to have deal with the last quantities. 
	Let us begin with  \eqref{f:T_2^k_E}
	% for 
	because  its 
	simplicity and to 
	%obtain  
	have 
	the basis of the induction. 
	From the last bound we see that $\E(A,B,C,D) \le \E(A,B) \E(C,D)$. 
	Further, 
	%we have 
	using the Cauchy--Schwarz inequality again, we have 
	\[
	\E^2(A,B) = \E(A,B,A,B)^2 = |\{ a b^{-1} = \bar{a} \bar{b}^{-1} ~:~ a,\bar{a} \in A,\, b,\bar{b} \in B \}|^2 
	= 
	\]
	\[
	= |\{ \bar{a}^{-1} a = \bar{b}^{-1} b ~:~ a,\bar{a} \in A,\, b,\bar{b} \in B\}|^2 \le \E(A) \E(B)
	\] 
	as required.
	Clearly, the same is true for functions. 
	%	Having 
	%$\T_{l} (A_1,\dots, A_{l}, A_1,\dots, A_{l})$.  
	%	$\T^2_{l} (A_1,\dots, A_{2l})$ we split $A_1,\dots, A_{2l}$ onto four parts $A,B,C,D$ of the same size of consecutive elements  and use the previous arguments for $E(A,B,C,D)$ (or, alternatively, split )
	Now let $k=2s$ is even (for odd $k$ a similar arguments hold). 
	Using  induction we obtain 
	\begin{equation}\label{tmp:30.04_1}
	\T^{k}_{k} (A_1,\dots, A_{2k}) = \T^k_{k/2} (r_{A_1 A^{-1}_2}, \dots, r_{A_{2k-1} A^{-1}_{2k}}) 
	\le 
	\prod_{j=1}^{k} \T_{k/2} (r_{A_{2j-1} A^{-1}_{2j}}) 
	= \prod_{j=1}^{k} \T_{k} (A_{2j-1}, A_{2j}) 
	%\,.
	\end{equation}
	and hence it is enough to prove for any $A$ and $B$ that
	\[
	\T^2_{k} (A,B) \le \T_{k} (A) \T_{k} (B) \,.
	\]
	%	If $k$ is odd, then just join, say,  $A_{2k}$ in def:T_k}
	%\[
	%	\T^{k/2}_{k} (A_1,\dots, A_{2k}) \le \T_k (A_{2k}) \prod_{j=1}^{k/2} \T_{k} (A_{2j},A_{2j+1}) 	
	%\]
	Now rewrite $\T_k (A,B) = \T_{2s} (A,B)$ as 
	\[
	(\bar{a}^{-1}_1 a_1) b^{-1}_1 a_2 b^{-1}_2 \dots a_s = \bar{b}^{-1}_1 \bar{a}_2 \bar{b}^{-1}_2 \dots \bar{a}_s (\bar{b}^{-1}_s b_s) 
	\]
	and using induction again and the arguments as in \eqref{tmp:30.04_1}, we obtain
	\[
	\T^k_{k} (A,B) \le \T_k (A) \T_k (B) \T^{k-2}_k (B,A) 
	%= \T_k (A) \T_k (B) \T^{k-2}_k (A,B) 
	\]
	and, similarly,
	\[
	\T^k_k (B,A) \le \T_k (A) \T_k (B) \T^{k-2}_k (A,B) \,.
	\]
	Combining the last two formulae, we get
	\[
	\T^{k^2} (A,B) \le \T^k_k (A) \T^k_k (B) \left( \T_k (A) \T_k (B) \T^{k-2}_k (A,B) \right)^{k-2} 
	=
	\T^{2k-2}_k (A) \T^{2k-2}_k (B) \T^{(k-2)^2}_k (A,B) 
	\]
	as required.
	%	This completes the proof.
	$\hfill\Box$
\end{proof}

%\bigskip 

\begin{corollary}
	The formula $\| f \| := \T_k (f)^{1/2k}$, $k\ge 2$ defines a norm of an arbitrary function $f: \Gr \to \C$.   
	Also, $\T_k (f)^{1/2k} \ge \| f\|_{2k}$. 
	\label{c:non-abelian_norm}
\end{corollary}

%\bigskip 

We need a well--known lemma, which we prove for the sake of completeness.

\begin{lemma}
	Let $\Gr$ be a group and let $\Gr$ acts $k$--transitively   on a set $X$.
	Suppose that $G\subseteq \Gr$ and $A,B\subseteq X$ are sets.
	Then
	\begin{equation}\label{f:k-transitive}
	\sum_{g\in G} \sum_{x\in B} A(gx) \le |G|^{1-\frac{1}{k}} |A||B| + |G| \,.
	\end{equation}
	\label{l:k-transitive} 
\end{lemma}
\begin{proof}
	Using the H\"older inequality, we get
	\[
	\sigma^k := 
	\left( \sum_{g\in G} \sum_{x\in B} A(gx) \right)^k \le |G|^{k-1} \sum_{g\in G} \left( \sum_{x\in B} A(gx) \right)^k
	=
	|G|^{k-1} \sum_{x_1,\dots,x_k \in B}\, \sum_{g\in G} A(g x_1) \dots A(g x_k) \,. 
	\]
	By the assumption $\Gr$ acts $k$--transitively  on $X$.
	Hence fixing $(x_1,\dots,x_k) \in B^k$ and $(x_1,\dots,x_k) \neq (a_1,\dots,a_k) \in A^k$, we find a unique $g\in \Gr$ such that $a_j = g x_j$, $j\in [k]$. 
	Thus
	\[
	\sigma^k \le |G|^{k-1} |A|^k |B|^k + |G|^{k-1} \sigma \,. 
	\]
	It gives us
	\[
	\sigma \le |G|^{1-\frac{1}{k}} |A||B| + |G| 
	\]
	as required.
	%	This completes the proof.
	$\hfill\Box$
\end{proof}

\bigskip

The well--known 
%general 
"counting lemma"\,
% estimate 
for general actions was proved many times, see, e.g., \cite{B_hyp} or \cite[Lemma 53]{sh_as}. 
We recall the proof for the case of completeness and because we will use some parts of the proofs later. 
%As we will see $2^k$ can be replaced to any even integer $n\ge 2$ for an arbitrary finite group.
Also, we replace $2^k$ in \eqref{f:actions}  to any even integer $n\ge 2$ for an arbitrary finite group.

\begin{lemma}
	Let $\Gr$ be a group, which acts on a set $X$ and let $f_1, f_2 : X \to \C$ be functions. 
	Also, let $G \subset \Gr$ be a set. 
	Then  for any $k\ge 1$, we get 
	\begin{equation}\label{f:actions}
	\left| \sum_{g \in G} \sum_x f_1 (x) f_2 (g x) \right|^{2^k} 
	%\frac{|G| \langle  f_1 \rangle \langle  f_2 \rangle }{|X|}
	\le 
	\| f_1\|^{2^{k}}_2 \| f_2\|^{2^{k}-2}_2 \cdot \sum_g r_{(GG^{-1})^{2^{k-1}}} (g) \sum_x f_2 (x) \overline{f_2 (g x)} \,.
	\end{equation}
	The same is true in the case $|\Gr| < \infty$ if one replaces $2^k$ to any nonzero even integer.  
	\label{l:actions}
\end{lemma}
\begin{proof}
	Denote by $\sigma$ the left--hand side of \eqref{f:actions}.
	Using the Cauchy--Schwarz, we obtain
	\begin{equation}\label{tmp:10.04_1}
	|\sigma|^2 \le \| f_1\|^2_2 \cdot \sum_x  \left| \sum_{g \in G} f_2 (g x) \right|^2
	=
	\| f_1\|^2_2  \cdot  \sum_g r_{GG^{-1}} (g) \sum_x f_2 (x) \overline{f_2 (g x)} \,. 
	\end{equation}
	Continuing this way, we get
	\[
	|\sigma|^{2^k} \le \| f_1\|^{2^{k}}_2 \| f_2\|^{2^{k}-2}_2 \cdot \sum_g r_{GG^{-1} \dots GG^{-1}} (g) \sum_x f_2 (x) \overline{f_2 (g x)} \,,
	\]
	where the term $GG^{-1}$ in $r_{GG^{-1} \dots GG^{-1}} (g)$ is taken $2^{k-1}$ times.
	Thus, \eqref{f:actions} follows.

	Let us give another proof for even powers and finite group $\Gr$.  
	%and firstly, we begin with the even powers.
	Returning to \eqref{tmp:10.04_1}, we have   
	\[
	|\sigma|^2 \le |\Gr|^{-1} \| f_1\|^2_2 \sum_{g,h} r_{GG^{-1}} (gh^{-1}) \sum_x f_2 (h x) \overline{f_2 (g x)} \,. 
	\]
	Consider a hermitian nonnegatively defined operator
	\begin{equation}\label{def:T}
	\T (g,h) = r_{GG^{-1}} (gh^{-1}) = \sum_{\a =1}^{|\Gr|} \mu_\a \_phi_\a (g) \ov{\_phi}_\a(h) \,, 
	\end{equation}
	where $\mu_\a \ge 0$ are eigenvalues and $\_phi_\a$ are correspondent eigenfunctions. 
	Thus
	\[
	|\sigma|^2 \le |\Gr|^{-1} \| f_1\|^2_2 \sum_{\a =1}^{|\Gr|} \mu_\a \left( \sum_x \left| \sum_{h} f_2 (h x) \ov{\_phi}_\a(h) \right|^2 \right)  \,.
	\] 
	Using the H\"older inequality and the orthogonality of the functions $\_phi_\a (g)$, we obtain
	\[
	|\sigma|^{2k} \le |\Gr|^{-k} \| f_1\|^{2k}_2 \sum_{\a =1}^{|\Gr|} \mu^k_\a \left( \sum_x \left| \sum_{h} f_2 (h x) \ov{\_phi}_\a(h) \right|^2 \right) 
	\cdot 
	\left( \sum_{\a =1}^{|\Gr|} \sum_x \left| \sum_{h} f_2 (h x) \ov{\_phi}_\a(h) \right|^2 \right)^{k-1}
	=
	\]
	\[
	= |\Gr|^{-k+1} \| f_1\|^{2k}_2  \sum_g r_{(GG^{-1})^k} (g) \sum_x f_2 (x) \overline{f_2 (g x)} 
	\cdot \left( \sum_h \sum_x |f_2 (h x)|^2 \right)^{k-1}
	= 
	\]
	\[
	=
	\| f_1\|^{2k}_2 \| f_2 \|^{2k-2}_2 \cdot \sum_g r_{(GG^{-1})^k} (g) \sum_x f_2 (x) \overline{f_2 (g x)} \,.
	\]
	%	It gives us the proof for the even powers.
	%	For arbitrary power we consider a decomposition of the operator $G(gh^{-1})$  similar to \eqref{def:T}. 
	%	%the singular decomposition 
	%%	\[
	%%		G(gh^{-1}) = \sum_{\a =1}^{|\Gr|} \mu^{1/2}_\a u_\a (g) \ov{v}_\a(h) \,.
	%%	\]
	%	After that we use \eqref{tmp:10.04_1} again as well as the fact that $\mu_\a \ge 0$  and apply the previous arguments.
	%	It requires to notice that anyway 
	%	and
	%\[
	%	\sum_{\a =1}^{|\Gr|} \mu^{n/2}_\a \left( \sum_x \left| \sum_{h} f_2 (h x) \ov{\_phi}_\a(h) \right|^2 \right)
	%\]
	This completes the proof.
	$\hfill\Box$
\end{proof}

\begin{remark}
	In terms of the eigenfunctions of the operator $\T$ from \eqref{def:T}, we have the following formula (let $|\Gr| < \infty$ for simplicity)
	\begin{equation}\label{f:T_k_eigenvalues}
	\T_{k} (G) = |\Gr|^{-1} \sum_{\a =1}^{|\Gr|} \mu^k_\a = |\Gr|^{-1}  \tr (\T^k) \,,
	\end{equation}
	and, clearly, $\T^k (g,h) = r_{(GG^{-1})^k} (gh^{-1})$. 
	\label{r:T_k_non-abelian}
\end{remark}

\section{First results on  incidences for hyperbolas}
\label{sec:hyp_first}

%\bigskip
%$\hfill\Box$

Take any $\la \neq 0$ and consider our basic equation 
\[
(y-a) (b-x) = \lambda 
\]
or, in other words, 
\begin{eqnarray}\label{f:basic_eq}
y = a + \frac{\lambda}{b-x} = g x \,,
\end{eqnarray}
where 
\[
u_a v_b = 
\left( {\begin{array}{cc}
	1 & a \\
	0 & 1 \\
	\end{array} } \right)
\left( {\begin{array}{cc}
	0 & \lambda \\
	-1 & b  \\
	\end{array} } \right)
=
\left( {\begin{array}{cc}
	-a & a b + \lambda \\
	-1 & b \\
	\end{array} } \right) = g \in G(A,B)  = G_\la (A,B) \,.
\]
Clearly, $\det (g) = \la \neq 0$ and hence in our main case  $\la=1$ we have   $G_1 (A,B) \subseteq \SL_2 (\F)$.  
Also, in the next Section we will consider the set 
\[G(A) =  \{ u_a v_a ~:~ a\in A \} \subseteq G(A,A) \,. \] 
Notice that $u_{a_1} u_{a_2} = u_{a_1+a_2}$ ("u"\, for a unipotent matrix from $\SL_2 (\F)$) and 
\[
v^{-1}_b = \lambda^{-1}
\left( {\begin{array}{cc}
	b & -\lambda \\
	1 & 0  \\
	\end{array} } \right)
\quad 
\mbox{ and }
\quad 
v_{b_1} v^{-1}_{b_2} = 
\left( {\begin{array}{cc}
	1 & 0 \\
	\la^{-1} (b_1-b_2) & 1  \\
	\end{array} } \right) = u^*_{\la^{-1} (b_1-b_2)} \in \SL_2 (\F) 
\,.
\]

\bigskip

Lemma below shows the connection between energy of a subset of $\SL_2 (\F)$ and the sum--product phenomenon. 
Formulae  \eqref{f:T_3(G)}, \eqref{f:T_2(G)} say us that any nontrivial upper bound for linear incidences in an arbitrary  field $\F$ implies a good upper estimate for  
$\T_2 (G_\la (A,B))$, $\T_3 (G_\la (A,B))$.

\begin{lemma}
	For any $A,B\subseteq \F$ and $\la \in \F$, $\la \neq 0$ one has
	\begin{equation}\label{f:T_3(G)}
	\T_3 (G_\la (A,B)) \le |A| |B| \sum_x r^{2}_{(A-A)(B-B)} (x) + |A|^4 |B|^4 \,.
	\end{equation}
	Besides 
	\begin{equation}\label{f:T_2(G)}
	\T_2 (G_\la (A,B)) \le |A|^2 \E^{+} (B) + |B|^2 \E^{+} (A) \,.
	\end{equation}
	%	Similarly, 
	\label{l:T_3(G)}
\end{lemma}
\begin{proof}
	Take three elements $g_1 = u_{a_1} v_{b_1}$, $g_2 = u_{a_2} v_{b_2}$, $g_3 = u_{a_3} v_{b_3}$ from $G_\la  (A,B)$. 
	Putting $\o_1 = \la^{-1} (b_1 - b_2)$, $\o_2 = a_3 - a_2$, we obtain
	\[
	g_1 g^{-1}_2 g_3 = u_{a_1} v_{b_1} v^{-1}_{b_2} u^{-1}_{a_2} u_{a_3} v_{b_3} = u_{a_1} u^*_{\o_1} u_{\o_2} v_{b_3} 
	=
	\]
	\[
	=
	\left( {\begin{array}{cc}
		-a_1 (\o_1 \o_2 +1) - \omega_2 & \la (1+a_1 \o_1) + b_3 (\o_2 + a_1 (1+\o_1 \o_2)) \\
		-(\o_1 \o_2 +1) & \la \o_1 + b_3 (\o_1 \o_2 +1)  \\
		\end{array} } \right)	\,.
	\] 
	If  $\o_1 \o_2 +1 \neq 0$, then 
	we reconstruct $a_1,b_3$, having the matrix above fixed. 
	%Finally, 
	Now it remains to notice  that
	\[
	|\{ \o_1 \o_2 = \o'_1 \o'_2 ~:~ \o_1 = \la^{-1} (b_1 - b_2), \o'_1 = \la^{-1} (b'_1 - b'_2), \o_2 = a_3 - a_2, \o'_2 = a'_3 - a'_2 \}| 
	=
	\]
	\[
	= 
	\sum_{x} r^{2}_{(A-A)(B-B)} (x) \,.  
	\]
	%as required.
	%	But if $\o_1 \o_2 +1 = 0$, then we just add the term $r^{2}_{(A-A)(B-B)} (-1)$ in the formula above. 
	%Now, 
	Further 
	if  $\o_1 \o_2 +1 = 0$, then $\o_1, \o_2 \neq 0$ and we can find, say, 	$a_1$, having $\o_1,\o_2$, $b_3$ and the matrix above fixed (see the right--up corner of the matrix above).
	%  are arbitrary and 
	Hence we need to count an additional term, which is at most 
	\[
	|A| |B|^2 \sum_{x\neq 0} r^{2}_{A-A} (x) r^{2}_{B-B} (-\la x^{-1}) \le |A| |B|^4  \E^{+} (A) \le |A|^4 |B|^4 \,.
	\]

	Similarly, to calculate $\T_2 (G_\la (A,B))$, we see that
	\[
	g_1 g^{-1}_2 = u_{a_1} u^*_{\o_1} u_{-a_2} 
	=
	\left( {\begin{array}{cc}
		1+a_1 \o_1 & a_1 - a_2 - a_1 a_2 \o_1 \\
		\o_1 & 1 - a_2 \o_1 \\
		\end{array} } \right)
	\] 
	and hence 
	\[
	\T_2 (G_\la  (A,B)) = |A|^2 (\E^{+} (B) - |B|^2) + |B|^2 \E^{+} (A) \le |A|^2 \E^{+} (B) + |B|^2 \E^{+} (A) \,.
	\]
	This completes the proof.
	$\hfill\Box$
\end{proof}

\begin{remark}
	Similarly, one can calculate higher energies of the set $\E^R_k (G_\la (A,B))$, $\E^L_k (G_\la (A,B))$ and prove  
	%Finally, by the same arguments, we have 
	\[
	\E^R_k (G_\la (A,B)) = |A|^2 (\E^{+}_k (B) - |B|^k) + |B|^k \E^{+}_k (A) \le |A|^2 \E^{+}_k (B) + |B|^k \E^{+}_k (A)
	\]  
	and 
	\[
	\E^L_k (G_\la (A,B)) \le |B|^2 \E^{+}_k (A) + |A|^k \E^{+}_k (B) \,.
	\]  	
\end{remark}

\bigskip

Using these upper bounds for the energy of the set $G(A,B)$, we obtain our first incidence result. 
Theorem \ref{t:hyp_incidences} implies Theorem \ref{t:hyp_incidences_intr} from the Introduction if one 
%use 
applies 
a trivial bound $\E^{+} (B,C) \le (|B| |C|)^{3/2}$.
Further, 
the first bound of Theorem \ref{t:hyp_incidences} is nontrivial only if $\E^{+} (C) \le |C|^{3-\eps}$ and $\E^{+} (B) \le |B|^{3-\eps}$, where $\eps>0$ 
but the second one is always nontrivial. 
Nevertheless it is interesting that incidences for hyperbolas are connected with the ordinary additive energy of a set.   
Also, the first bound takes place in any field not only in $\F_p$.

\begin{theorem}
	Let $A,B,C,D\subseteq \F_p$ be sets.
	Then for any $\la \neq 0$, one has 
	\begin{equation}\label{f:hyp_incidences_def}
	|\{ (a+b) (c+d) = \lambda ~:~ a\in A,\, b\in B,\, c\in C,\, d\in D \}| - \frac{|A||B||C||D|}{p} 
	\lesssim 
	\end{equation}
	\[
	\lesssim \min \{ |D|^{1/2} |B| |C| + |A| |D|^{1/2} (|B| |C|)^{1/3} (|B|^{1/3} \E^{+} (C)^{1/6} + |C|^{1/3} \E^{+} (B)^{1/6}), 
	%|A| |C|^{1/3} |B|^{2/3} |D|^{1/2} \E^{+} (C)^{1/6}, 
	\]
	\begin{equation}\label{f:hyp_incidences}
	|A|^{1/4} |B||C| |D|^{1/2} + |A|^{3/4} (|B||C|)^{19/24} |D|^{1/2} (\E^{+} (B,C))^{1/24} \} \,.
	\end{equation}
	%	Besides the quantity from \eqref{f:hyp_incidences_def} can be estimated as  
	%\[
	%	|A|^{1/4} |B||C| |D|^{1/2} + |A|^{3/4} |D|^{1/2} (|B||C|)^{5/8} (\E^{+} (B,C))^{1/24} \,.
	%	%\left( |C| \E^{+} (B)^{1/2} + |B| \E^{+} (C)^{1/2} \right)^{1/18} \,.
	%\]
	\label{t:hyp_incidences}
\end{theorem}
\begin{proof}
	Rewrite our basic equation $(a+b) (c+d) = \lambda$ as $(y-a) (b-x) = \la$, where we have new variables $y\in A$, $x\in -D$, $a\in -B$, $b\in C$. 
	In other words, we 
	%are interested in 
	need to count the number of the  solutions $\sigma$ to the equation 
	$g x =y$ with $g\in G_\la (-B,C) := G$ and  $y\in A$, $x\in -D$. 
	Let $f_1 (x) = D(-x) - |D|/p$. Then 
	\[
	\sigma = \frac{|A||B||C||D|}{p} + \sum_{g\in G} \sum_x f_1 (x) A(gx) = \frac{|A||B||C||D|}{p} + \sigma_* \,.
	\]
	Here one can consider other balanced functions, e.g., of the set $A$ or even of the sets $(-B)$, $C$ in our  set of actions  $G$
	%(that is 
	(in other words $g$ is taken with the correspondent weight in this case).  
	Using  Lemma \ref{l:actions} with $k=1$, we get
	\[
	\sigma^2_* \le |D| \sum_g r_{GG^{-1}} (g) \sum_{x \in A} A(gx) \,. 
	\]
	Applying Lemma \ref{l:k-transitive} with $k=3$, as well as the second part of Lemma \ref{l:T_3(G)}, we obtain
	\[
	\sigma_* \ll |D|^{1/2} |B| |C| + |A| |D|^{1/2} (|B| |C|)^{1/3} \T^{1/6}_2 (G) 
	%\le 2 |A| |D|^{1/2} (|B| |C|)^{2/3} \T^{1/6}_2 (G) 
	\le
	\]
	\[
	\le
	|D|^{1/2} |B| |C| + |A| |D|^{1/2} (|B| |C|)^{1/3} (|B|^{1/3} \E^{+} (C)^{1/6} + |C|^{1/3} \E^{+} (B)^{1/6}) \,.
	\]
	%\[
	%	\le
	%	2 |A| |B| |D|^{1/2} \E^{+} (C)^{1/6} \,. 
	%\] 

	Similarly, using Lemma \ref{l:actions} 
	with $k=2$, 
	we 
	%obtain
	have 
	\begin{equation}\label{tmp:12.04_1}
	\sigma^4_* \le |D|^2 |A| \sum_g r_{GG^{-1} G G^{-1}} (g) \sum_{x \in A} A(gx) \,.
	\end{equation}
	Denote by $w(g) =  \sum_{x \in A} A(gx)$.
	% and return to \eqref{tmp:12.04_1}. 
	It gives us
	\[
	\sigma^4_* \le |D|^2 |A| \sum_g r_{GG^{-1} G} (g) r_{wG} (g) \,,
	\]
	and by the pigeonhole principle there is $\tau$ such that 
	\begin{equation}\label{tmp:12.04_2}
	\sigma^4_* \lesssim |D|^2 |A| \tau \sum_g r_{GG^{-1} G} (g) r_{S_\tau G} (g) \,,
	\end{equation}
	where $S_\tau = \{ g\in \SL_2 (\F_p) ~:~ \tau \le w(g) \le 2\tau \}$. 
	From the last inequality one can 
	%see that 
	derive 
	$\tau \gg |D|^{-2} |A|^{-1} |G|^{-4} \sigma^4_*$.
	It follows that if $\tau \ll 1$, then $\sigma_* \ll |G| |D|^{1/2} |A|^{1/4}$.  
	Otherwise in view of Lemma \ref{l:k-transitive}, we 
	%obtain 
	have 
	$|S_\tau| \ll |A|^6/\tau^3$. 
	Now combining \eqref{tmp:12.04_2}, 
	%Lemma \ref{l:T_2^k} with $k=1$, 
	the second part of Lemma \ref{l:T_2^k} and the Cauchy--Schwarz inequality, we obtain 
	\[
	\sigma^8_* \lesssim  
	|D|^4 |A|^2 \tau^2 \T_3 (G) \E (G, S^{-1}_\tau, G, S^{-1}_\tau)
	\le |D|^4 |A|^2 \tau^2 \T_3 (G) |G|^2 |S_\tau| \,.
	% \le |D|^4 |A|^2 \tau^2 \T_3 (G) \E(G)^{1/2} |S_\tau|^{3/2} \,.	
	\]
	Using $|S_\tau| \ll |A|^6/\tau^3$ and our lower bound 	$\tau \gg |D|^{-2} |A|^{-1} |G|^{-4} \sigma^4_*$, we get
	\begin{equation}\label{tmp:12.04_3}
	\sigma^{12}_* \cdot (|D|^{-2} |A|^{-1} |G|^{-4})^{} \ll \sigma^8_* \tau^{} \lesssim |D|^4 |A|^{8} |G|^2 \T_3 (G) \,.	
	\end{equation}
	Applying Lemma \ref{l:T_3(G)} and Lemma \ref{l:D_2} to estimate $\T_3 (G)$, we derive
	\begin{equation}\label{tmp:06.04_0}
	\T_3 (G) \le |B| |C| \sum_x r^{2}_{(B-B)(C-C)} (x) + |B|^4 |C|^4 
	\lesssim 
	(|B||C|)^{7/2} (\E^{+} (B,C))^{1/2} \,.
	\end{equation}
	Here we do not need to have deal the term $(|B||C|)^4/p$ in Lemma \ref{l:D_2} because one can consider the balanced function of  
	$(-B)$, $C$ in the set of actions  $G$ (see details in \cite{sh_as}).  
	%It follows that as well as the previous calculations, we have 
	Combining the last inequality with \eqref{tmp:12.04_3}, we obtain  
	\[
	\sigma^{12}_* \lesssim  |D|^6 (|B||C|)^{6} |A|^{9} (|B||C|)^{7/2} \E^{+} (B,C)^{1/2}  \,.
	\]
	This completes the proof.
	%In view of  the first part of Lemma \ref{l:T_3(G)} and Lemma \ref{l:D_2}, we have 
	%\[
	%\sigma_* \ll |A|^{1/4} |B||C| |D|^{1/2} + |A|^{3/4} (|B||C|)^{19/24} |D|^{1/2} (\E^{+} (B,C))^{1/24}
	%\]
	%as required.
	$\hfill\Box$
\end{proof}

\begin{remark}
	One can apply general results from \cite{NG_S}, \cite{sh_as} to nontrivially estimate  $\T_4 (G)$ via $\T_2 (G)$ in Theorem \ref{t:hyp_incidences} (see formula \eqref{tmp:12.04_1})	but we prefer to use $\T_3 (G)$ because it gives better bounds.  
\end{remark}

Using a trivial bound $\E^{+} (B,C) \le (|B||C|)^{3/2}$, we obtain

\begin{corollary}
	Let $A,B,C,D\subseteq \F_p$ be sets.
	Then for any $\la \neq 0$, one has 
	\[
	|\{ (a+b) (c+d) = \lambda ~:~ a\in A,\, b\in B,\, c\in C,\, d\in D \}| - \frac{|A||B||C||D|}{p} 
	\lesssim 
	\]
	\begin{equation}\label{fc:hyp_incidences}
	\lesssim 
	|A|^{1/4} |B||C| |D|^{1/2} + |A|^{3/4} (|B||C|)^{41/48} |D|^{1/2}  \,.
	\end{equation}
	\label{c:hyp_incidences}	
\end{corollary}

\begin{remark}
	In \cite[Theorem 41]{sh_as} it was proved, in particular, that for any $c<1/192$  one has 
	%the quantity 
	$\sum_x r^{2}_{(A-A)(A-A)} (x) \ll |A|^{13/2-c}$, provided 
	$|A|\le p^{48/97}$. 
	It gives an improvement of Theorem \ref{t:hyp_incidences} at least in the symmetric case (in particular, it gives a better bound in formula \eqref{f:r_AA} of Corollary \ref{c:r_AA} below for small $A$).  
\end{remark}

Now let us obtain an upper bound for size of hyperbola with elements from a set with small sumset. 
First results of this type were obtained in \cite{NG_S} but our new bound is more "quantitative".

\begin{corollary}
	Let $A\subseteq \F_p$ be a set.
	Suppose that $|A+A| \le K |A|$. 
	Then for any $\la \neq 0$, one has 
	\begin{equation}\label{f:r_AA}
	r_{AA} (\la) \lesssim \frac{K^2 |A|^2}{p} + K^{5/4} |A|^{23/24} \,.
	\end{equation}
	Finally, if 
	%$|A| \le p^{13/23} K^{50/23}$, 
	%	$K^{17} |A|^{92} \le p^{52}$, 
	$|A-A|^{92} \le p^{52}$, 
	then
	\begin{equation}\label{f:r_AA'}
	r_{AA} (\la) \lesssim \frac{K^2 |A|^2}{p} + O_K (|A|^{149/156}) \,.
	\end{equation}
	\label{c:r_AA}
\end{corollary}
\begin{proof}
	Put $S=A+A$. 
	We have $A+x \subset A+A$ for any $x\in A$. 
	Hence 
	%\begin{equation}\label{tmp:06.04_1}
	%	r_{AA} (\la) |A|^2 \le |\{ (a+b) (c+d) = \lambda ~:~ a\in A+A,\, b\in -A,\, c\in A+A,\, d\in -A \}| \,.
	%\end{equation}
	\begin{equation}\label{tmp:06.04_1}
	r_{AA} (\la) |A|^2 \le |\{ (a+b) (c+d) = \lambda ~:~ a\in A+A,\, b\in -A,\, c\in -A,\, d\in A+A \}| \,.
	\end{equation}
	Applying the second part of Theorem \ref{t:hyp_incidences}, as well as a trivial  estimate 
	%$\E^{+} (-A,A+A) \le K|A|^3$, we obtain
	$\E^{+} (-A,-A) \le |A|^3$, we 
	%obtain
	get 
	\[
	r_{AA} (\la) \lesssim  \frac{K^2 |A|^2}{p} + K^{3/4} |A|^{3/4} + K^{5/4} |A|^{23/24} \ll \frac{K^2 |A|^2}{p} + K^{5/4} |A|^{23/24} \,.
	\]
	%	as required.

	To obtain the second bound of Corollary \ref{c:r_AA} we apply estimate \eqref{tmp:06.04_1} and then we use the first part of Lemma \ref{l:T_3(G)} directly. 
	It gives us (see formulae \eqref{tmp:12.04_3}, \eqref{tmp:06.04_0} from the proof of Theorem \ref{t:hyp_incidences})
	\begin{equation}\label{tmp:06.04_2}
	r_{AA} (\la) \ll \frac{K^2 |A|^2}{p} + K^{3/4} |A|^{3/4} + (K |A|)^{5/4}  |A|^{-5/6} \left( \sum_x r^2_{(A-A)(A-A)} (x) \right)^{1/12} \,.
	\end{equation}
	To estimate $\sum_x r^2_{(A-A)(A-A)} (x)$, 
	%we use trivial bounds and 
	%have
	%	obtain 
	we use the Pl\"unnecke inequality \cite{TV}, combining with Lemma \ref{l:E_times_A}, and obtain
	\[
	\sum_x r^2_{(A-A)(A-A)} (x) \le 
	%|A|^2 |S|^2 (\E^{+} (A-A) \E^{+} (2A-2A) )^{1/2} 
	|A|^4 \E^{\times} (A-A) 
	\ll_K
	|A|^{4 +32/13} \,,
	\]  
	provided 
	$|A-A|^{117} \le p^{52} |2A-2A|^{25}$. 
	The last inequality satisfies thanks to our condition 	$|A-A|^{92} \le p^{52}$.
	%	 $|A+A|^{17} |A|^{75} \le p^{52}$ as well as   the Pl\"unnecke inequality. 
	This completes the proof.
	$\hfill\Box$
\end{proof}

\bigskip

To compare, using the Szemer\'edi--Trotter Theorem \cite{ST}, one can obtain $r_{AA} (\la) \ll_K |A|^{2/3}$ for any finite $A\subset \R$ with $|A+A| \le K|A|$ and $\la \neq 0$.

\bigskip

%Another 
In our final 
consequence of Theorem \ref{t:hyp_incidences} we have deal with the case of arithmetic progressions.

\begin{corollary}
	Let $A,B,C,D\subseteq \F_p$ be sets and $B$, $C$ be arithmetic progressions with the differences equal one.
	Then for any $\la \neq 0$, one has 
	\[
	|\{ (a+b) (c+d) = \lambda ~:~ a\in A,\, b\in B,\, c\in C,\, d\in D \}| - \frac{|A||B||C||D|}{p} 
	\lesssim
	\]
	\begin{equation}\label{fc:hyp_incidences_progr}
	\lesssim 
	|A|^{1/4} |B||C| |D|^{1/2} + |A|^{3/4} |D|^{1/2} (|B||C|)^{5/6} \left( 1+ \left(\frac{|B||C|}{p} \right)^{1/12} \right)  \,.
	\end{equation}
	\label{c:hyp_incidences_progr}	
\end{corollary}
\begin{proof}
	Indeed, by Theorem \ref{t:E^*_progr}, we know that 
	% \cite[Theorem 2]{CSZ}, we have for any arithmetic progressions $B$ and $C$ 
	\[
	\E^\times (B,C) = \frac{|B|^2 |C|^2}{p} + O(|B||C| \log^2 p) \,.
	\]
	We have 
	\[
	\sum_x r^{2}_{(B-B)(C-C)} (x) \le (|B||C|)^2 \E^{\times} (B-B, C-C) \,.
	\]
	After that apply the arguments of the proof of Theorem \ref{t:hyp_incidences} and our upper bound for $\E^\times (B,C)$.
	It gives us, in particular, 
	%\begin{equation}\label{tmp:06.04_0'}
	\[
	\T_3 (G (-B,C)) \le |B| |C| \sum_x r^{2}_{(B-B)(C-C)} (x) + |B|^4 |C|^4 
	\ll
	\]
	\[
	%\ll
	(|B||C|)^3  \E^{\times} (B-B, C-C) + |B|^4 |C|^4 \ll (|B||C|)^3  \E^{\times} (B-B, C-C) 
	\ll
	%\]
	%\[
	%	 \ll
	\frac{(|B||C|)^5}{p} + (|B||C|)^4 \log^2 p 
	%\,.
	\]
	and after that we substitute this bound into estimate $\sigma_*$ as in Theorem \ref{t:hyp_incidences}. 
	This completes the proof.
	$\hfill\Box$
\end{proof}

\bigskip 

In a natural way, in the case of arithmetic progressions one can try to estimate higher energies $\T_k (G_\la (B,C))$. 
It turns out that in this situation "first--stage"\,  methods \cite{BG} work rather good and it will be done in the next Section, 
see 
%Theorems \ref{t:B_progr}, \ref{t:B_progr_new} as well as 
Proposition \ref{p:A_abs_Re} and Theorem \ref{t:A_abs_Re_Z}  below.

\section{Asymmetric results}
\label{sec:hyp_second}

%\bigskip
%$\hfill\Box$

In \cite{NG_S} and \cite{sh_as} the authors obtain a series of upper bounds for equation \eqref{f:hyp} in asymmetric cases (i.e. when the set of actions is relatively small). 
Let us recall two results from these papers.

\begin{theorem}
	Let $\la\in \F_p^*$, $f_1, f_2 : \F_p \to \C$ be functions, and $S,T \subseteq \F_p$ be sets. 
	Also, let $G = G_\la (S,T)$ and $|S| |T| \ge p^\eps$.  
	Then  there is $\d = \d (\eps) >0$ such that
	\begin{equation}\label{f:B_any}
	\sum_{g \in G} \sum_x f_1 (x) f_2 (g x) - p^{-1} \left(\sum_x f_1 (x) \right) \cdot \left(\sum_x f_2 (x) \right)
	\le
	2\|f_1 \|_2 \|f_2 \|_2 |S||T| p^{-\d} \,. 
	\end{equation}
	Further, if $A,B,C,D \subseteq \F_p$ are any sets with $|B||C| \ge (|A||D|)^\eps$ and  $\la \in \F_p^*$, then  
	\[
	|\{ (a+b) (c+d) = \lambda ~:~ a\in A,\, b\in B,\, c\in C,\, d\in D \}| - \frac{|A||B||C||D|}{p} 
	\ll
	\]
	\begin{equation}\label{f:B_any_A}
	\ll
	(|A||D|)^{1/2} (|B||C|)^{1-\d} \,. 
	\end{equation}
	\label{t:B_any}
\end{theorem}

The first part of Theorem \ref{t:B_any} is Lemma 53 from \cite{sh_as} and the second part follows by the same arguments (with $S=B$, $T=C$, $f_1(x) = A(x) - |A|/p$, $f_2 (x) = D(x)$, say) 
if one uses, in addition, any rough incidence result in $\SL_2 (\F_p)$ see, e.g.,   \cite{Brendan_rich} or  our Lemma \ref{l:k-transitive}. 
%say, 
%, e.g.,
%or
%our 
% Section \ref{sec:hyp_first}. 

Now if the sets $B$, $C$ in our sets $G_\la (B,C)$, $G_\la (B)$ are special, then one can improve Theorem \ref{t:B_any} as was done in [Theorem 11]\cite{NG_S} (or see the sketch of the proof of Theorem \ref{t:B_progr_new} below).

\begin{theorem}
	Let 
	%$\la \in \F_p^*$, 
	$A,D \subseteq \F_p$ be any sets,
	% $|A|=|D|$ 
	and 
	$B = 2\cdot [N]$, 
	%	$B = [N]$,
	% $C=[N]$,  
	%$5 \le 
	%$m:= \min\{N,M\} 
	$N \le p^{\tau}$, $0<\tau < 1/8$. 
	Suppose that $|D| \le p^{1-\d}$, $\d = C_1^{-1/\tau}$.
	Then  
	\begin{equation}\label{f:B_progr}
	|\{ (a+b) (b+d) = 1 ~:~ a\in A,\, b\in B,\, %c\in C,\, 
	d\in D \}| 
	\ll
	\sqrt{|A||D|} N \cdot N^{- C_2 \d/\tau} \,. 
	\end{equation}
	Here $C_1, C_2 >1$ are absolute constants.
	\label{t:B_progr}
\end{theorem}

\begin{corollary}
	Let $A \subseteq \F_p$ be a set.
	%,  and $N \le p^{\tau}$, $\tau <1/8$ be numbers. 
	Suppose that $|A|< p^{0.99}$ and $N\le p^{\tau_0}$, where $\tau_0 < 1/8$ is an absolute constant. 
	Then 
	%there are $i,j \in [N]$ 
	there is $i\in 2 \cdot [N]$ 
	such that 
	\begin{equation}\label{f:A_abs}
	\left|(A+i) \cap \frac{1}{A+i} \right| \ll 
	%\frac{|A|}{N^{c(\tau)}} \,,
	\frac{|A|}{N^{c}} \,,
	\end{equation}
	%	where $c (\tau)>0$ depends on $\tau$ only. 
	where $c>0$ is an absolute constant. 
	\label{c:A_abs} 
\end{corollary}

The proof of Theorem \ref{t:B_progr} uses the following result, see \cite[Lemma 27]{NG_S}.

%Finally, we recall a non--abelian result from \cite[Lemma 27]{NG_S}.

\begin{lemma}
	Let
	% $G=G(2\cdot[N],2\cdot[N])$. 
	$G=G(2\cdot[N]) \subset \SL_2 (\Z)$. 
	Then there is an absolute constant $C_*>0$  such that for an arbitrary integer $s$ the following holds
	\[
	\T_{s} (G) \ll C_*^{s} |G|^{s} \,. 
	\]
	In $\F_p$ the same is true for all $s$ such that 
	\begin{equation}\label{cond:Lemma 27}
	s\le \frac{1}{4} \log_N p \,.
	\end{equation} 
	\label{l:Lemma 27}
\end{lemma}

\begin{remark}
	Actually, one can check from the proof of \cite[Lemma 27]{NG_S}, see \cite[Theorem 25,29]{NG_S} that condition \eqref{cond:Lemma 27} can be replaced by 
	$s\le (\frac{1}{2} - \eps) \log_N p$ for any $\eps>0$ and sufficiently large $p$ and $N$.  
	This square root condition looks rather natural.  
	Further, in $\R$ one can consider  a more general family $G=G(\o \cdot[N]) \subset \SL_2 (\R)$, where $|\o|\ge 2$, see the proof of Theorem \ref{t:A_abs_Re_Z}. 
	\label{r:0.5}
\end{remark}

In the real setting one can easily calculate the constant $c$ from \eqref{f:A_abs} in a simple way.
Surprisingly, that our saving in the asymmetric case of sets of rather different 
%equal 
cardinality (when $|A|, |D|$ are large comparable to $N$, see below)
is better than the famous Szemer\'edi--Trotter Theorem gives us (of course it is because a pair of our sets are arithmetic progressions).
Although the focus of this paper is $\F_p$ we give a proof of this result here because its simplicity and because we will use some parts of the proof later.

\begin{proposition}
	Let $\la \neq 0$ be any number, $A \subset \R$ be a set and  let $N\ge 1$ be an integer. 
	%$N, M \ge 1$ be integers. 
	%$B\subset \Z$  be sets. 
	Then for $|\omega| \ge 2$ one has 
	\begin{equation}\label{f:A_abs_Re}
	|\{ (a+b) (b+d) = 1 ~:~ a\in A,\, d\in D,\,  b \in  \o \cdot [N] %[N] 
	\}| 
	\ll \sqrt{|A||D|} N \max\{ |D|^{-1/2}, N^{-1/5} \} \,.
	\end{equation}
	More precisely, if for a certain $l$ the following holds $|D|^2 \ge N^{l}$, then 
	\begin{equation}\label{f:A_abs_Re2}
	|\{ (a+b) (b+d) = 1 ~:~ a\in A,\, d\in D,\,  b\in \o \cdot [N] %[N] 
	\}| 
	\ll \sqrt{|A||D|} N^{2/3} \cdot |D|^{1/6l} \,.
	\end{equation}
	\label{p:A_abs_Re} 
\end{proposition}
\begin{proof}
	%	We consider a slightly more general situation $C=[M]$ because most of our arguments work in this case as well. Of course we will put $N=M$.   
	%	Multiplying \eqref{f:hyp} by two we can assume that, actually, $b \in 2\cdot [N]$.
	%, $c\in 2\cdot [M]$.
	Denote by $\sigma$ the left-hand side of \eqref{f:A_abs_Re} and let 
	%	$G=G(2\cdot [N], 2\cdot [M])$.  
	$G=G(\o \cdot [N])$.  
	Using Lemma \ref{l:actions}, we obtain that for any $k\ge 1$ the following holds 
	\[
	\sigma^{2k} \le |A|^{k} |D|^{k-1} \cdot \sum_g r_{(GG^{-1})^{k}} (g) \sum_x D (x) D (g x) \,.	
	\]
	Applying the Szemer\'edi--Trotter Theorem \cite{ST}, we see that either 
	\begin{equation}\label{tpm:07.04.2019_1}
	\sigma \ll |G| \sqrt{|A| |D|} |D|^{-1/2k}
	\end{equation} 
	or
	\begin{equation}\label{tpm:07.04.2019_2} 
	\sigma \ll |G|^{1/3} \sqrt{|A| |D|} |D|^{1/6k} \T_{2k} (G)^{1/6k} \,.
	\end{equation} 
	By Lemma \ref{l:Lemma 27} %(and here we use the fact that $N=M$) 
	(also, see  Remark \ref{r:0.5})
	we know that $\T_{s} (G) \ll C_*^{s} |G|^{s}$, where $C_*>0$ is an absolute constant and $s$ is an arbitrary integer. 
	Hence in the second case, we have
	\begin{equation}\label{tpm:07.04.2019_2+} 
	\sigma \ll |G|^{2/3} \sqrt{|A| |D|} |D|^{1/6k} \,.
	\end{equation}
	Now let $k$ be the first number such that \eqref{tpm:07.04.2019_1} takes place.
	We can assume that $k>1$ because otherwise we are done.
	Then \eqref{tpm:07.04.2019_2+} holds with $k-1\ge 1$. 
	Comparing  bounds \eqref{tpm:07.04.2019_1}, \eqref{tpm:07.04.2019_2+}, we obtain 
	\[
	\sigma \ll \sqrt{|A||D|} |G|^{\frac{3k-2}{4k-3}} \le \sqrt{|A||D|} |G|^{\frac{4}{5}} 
	%\,.
	\]
	as required. 
	To prove \eqref{f:A_abs_Re2} suppose that \eqref{tpm:07.04.2019_2+} does not hold with $k=l$. 
	Then by \eqref{tpm:07.04.2019_1}, we get
	\[
	|G|^{2/3} \sqrt{|A| |D|} |D|^{1/6l} \ll \sigma \ll |G| \sqrt{|A| |D|} |D|^{-1/2l}
	\]
	or, in other words, $|D| \ll |G|^{l/2} 
	%= (NM)^{l/2} 
	= N^{l/2}$ 
	and this contradicts with our assumption. 
	This completes the proof.
	$\hfill\Box$
\end{proof}

\bigskip 

\begin{example}
	Suppose that $Q$ is an arithmetic progression on length $2M$ and put\\  $A=Q \bigsqcup (\bigsqcup_{j=1}^N (Q^{-1} + j))$, $M\ge N$.
	Then it is easy to see that for any $i,j \in [N]$ the set  $(A-i) \cap (A-j)^{-1}$  contains $[M]$. 
	Also, choosing $Q$ in an appropriate way, we can assume that $|A| \sim N |Q|$. Hence our saving $1/5$ in  	\eqref{f:A_abs_Re} (or analogously  the saving ${1/3}$ in  \eqref{f:A_abs_Re2}) cannot be replaced by any number strictly 
	%less 
	greater 
	than $1$. 
\end{example}

Now we formulate an analogue of Proposition  \ref{p:A_abs_Re} for an arbitrary set $B \subset \Z$.

\begin{theorem}
	Let $A,D \subset \R$, $B,C\subset \Z$  be sets, and $\la \neq 0$ be any number.
	%, $|\la| \le |C|/4$. 
	Then 
	%	\begin{equation}\label{f:A_abs_Re_Z}
	\[
	|\{ (a+b) (c+d) = \lambda ~:~ a\in A,\, d\in D,\, b \in B,\, c\in C \}|
	\ll_{|\la|}
	\] 
	\begin{equation}\label{f:A_abs_Re_Z}
	\ll_{|\la|} 
	\sqrt{|A||D|} |B||C| \cdot \max\{ |D|^{-1/2}, \rho(B,C) \}  \,,
	\end{equation}
	%	where $c>0$ is an absolute constant. 
	where
	\[
	\rho(B,C) := \max_{k\ge 2} \left( \frac{\E^{+} (C)^{k-1} \E^{+} (B)^{k-2}}{|B|^{4k-6}|C|^{4k-4}}\right)^{1/(8k-6)} 
	\le
	(|B|^k |C|^{k-1})^{-\frac{1}{8k-6}} \,.
	\]
	More precisely, if for a certain $l$ the following holds $|D|^4 \ge |B|^{4l-2} |C|^{4l} \E^{+} (C)^{-l} \E^{+} (B)^{-l+1}$, then 
	\[
	|\{ (a+b) (c+d) = \lambda ~:~ a\in A,\, d\in D,\, b \in B,\, c \in C \}|
	\ll_{|\la|} 
	\]
	\begin{equation}\label{f:A_abs_Re_Z2}
	\ll_{|\la|} 
	(|B||C|)^{1/3} \sqrt{|A| |D|} \cdot |D|^{1/6l} (|B|^2 \E^{+} (C)^{l} \E^{+} (B)^{l-1})^{1/6l}  \,.
	\end{equation}
	\label{t:A_abs_Re_Z} 
\end{theorem}
\begin{proof}
	We use the arguments from the proof of Proposition \ref{p:A_abs_Re}.
	In particular, 	
	%using 
	applying 
	Lemma \ref{l:actions}, we see that 
	%we need 
	our aim is 
	to estimate the quantity $\T_{2k} (G_\la  (-B,C))$ as in  \eqref{tpm:07.04.2019_2} 
	but before we need  some preparations.  
	%	Multiplying \eqref{f:hyp} by two we can assume that, actually, $b,c \in 2\cdot \Z$.
	Put $m = 2\lceil |\la| \rceil$ and $G = G_\la (-B,C)$. 
	Split  $B$ onto odd/even numbers $B_0$, $B_1$, further, split $C \subset \Z$ 
	%or $B,C \subset \F_p$ it is just split this sets 
	onto congruence classes $C_j$ modulo $m$ and use Lemma \ref{l:T_2^k} or its consequence Corollary \ref{c:non-abelian_norm} to estimate $\T_{2k} (G)$ via $\T_{2k}$ on sets $G_\la (B_i,C_j)$.
	%  defined by these congruences.   
	In the notation of the beginning of Section \ref{sec:hyp_first}, we get
	\[
	g_1 g^{-1}_2 \dots g_{2k-1} g^{-1}_{2k} = u_{a_1} u^{*}_{\la^{-1} (b_2-b_1)} u_{a_3-a_2} \dots u^*_{\la^{-1} (b_{2k}-b_{2k-1})} u^{-1}_{a_{2k}} 
	=
	\bar{g}_1 \bar{g}^{-1}_2 \dots \bar{g}_{2k-1} \bar{g}^{-1}_{2k} 
	=
	\]
	\begin{equation}\label{f:eq_T}
	=
	u_{a'_1} u^{*}_{\la^{-1} (b'_2-b'_1)} u_{a'_3-a'_2} \dots u^*_{\la^{-1} (b'_{2k}-b'_{2k-1})} u^{-1}_{a'_{2k}}  \,,  
	\end{equation}
	where variables $a_j, a'_j$ and $b_j,b'_j$ are from $(-B),C$, correspondingly.   
	% \in B$. 
	%	Consider $\mu B$ instead of $B$ and $\mu A$ instead of $A$ and $\la \mu$ for any number $\mu$.
	%	Clearly, this does not change equation \eqref{f:A_abs_Re_Z}. 
	%	Now take $\mu = 2(\min_{b,b'\in B} |b-b'|)^{-1}$.
	%	Then the distance between elements of $\mu B$ is at least $2$. 
	%Another way to obtain a free subgroup in the case $B\subset \Z$ or $B \subset \F_p$, say,  is the following: 
	%	Now split our set $B$ as $B=B' \bigsqcup B''$, where $B'$ consists of even and $B''$ consists of odd numbers, correspondingly. 
	%	Write $G' = G(B',B')$ and $G'' = G(B'',B'')$. 
	%Then by Lemma \ref{l:T_2^k} it is enough to estimate  $\T_{2^k} (G (B',B'))$, $\T_{2^k} (G (B'',B'')).$
	% and both these quantities can be bounded similarly.
	%Indeed, 
	Further, 
	it is well--known (see, e.g., \cite{MKS_book}) that the matrices 
	\begin{equation}\label{f:matrices_st}
	\left( {\begin{array}{cc}
		1 & s \\
		0 & 1 \\
		\end{array} } \right)
	, \quad 
	\left( {\begin{array}{cc}
		1 & 0 \\
		t & 1 \\
		\end{array} } \right)	
	\end{equation}	
	generate a free subgroup of $\SL_2 (\Z)$, provided $|s|, |t| \ge 2$ or even when $|st| \ge 4$ (it easily follows from the ping--pong lemma).  
	%	Hence from \eqref{f:eq_T}, we see that 
	%	Another way to obtain a free subgroup in the case $B,C \subset \Z$ or $B,C \subset \F_p$ it is just split this sets onto odd/even numbers and use Lemma \ref{l:T_2^k}.   
	Rewriting \eqref{f:eq_T} as 
	\begin{equation}\label{f:I_free}
	(u^{*}_{\la^{-1} (b'_{2k}-b'_{2k-1})})^{-1} \dots (u^{*}_{\la^{-1} (b'_2-b'_1}))^{-1} u_{a_1-a'_1} u^{*}_{\la^{-1} (b_2-b_1)} u_{a_3-a_2} \dots u^*_{\la^{-1} (b_{2k}-b_{2k-1})} u_{a'_{2k-1}-a_{2k}}  = I 
	\end{equation}
	we see that the number solutions to the last equation is $|B|^2 (\E^{+} (C))^{k} (\E^{+} (B))^{k-1}$
	(for $k=1$ it coincides with the second bound of Lemma \ref{l:T_3(G)} in the symmetric case).
	Indeed, since $B_i-B_i \subset 2\cdot \Z$,  $C_j-C_j \subset m\cdot \Z$, it follows that all $u_a$ and $u^*_b$ are powers of matrices from \eqref{f:matrices_st} with $s=2$ and $|t| = |m/\la| \ge 2$, correspondingly, 
	and hence equation \eqref{f:I_free} has no nontrivial solutions. 
	Thus, as in \eqref{tpm:07.04.2019_1}, \eqref{tpm:07.04.2019_2}, we 
	%obtain 
	have 
	either 
	\begin{equation}\label{tpm:07.04.2019_1'}
	\sigma \ll |B||C| \sqrt{|A| |D|} |D|^{-1/2k}
	\end{equation} 
	or
	%, using Lemma \ref{l:k-transitive}, we get 
	%\[
	%	\sigma^{6k} \ll |A|^{3k} |D|^{3k+1} |G|^{2k} \T_{2k} (G) \,.
	%\]
	%In other words,
	\begin{equation}\label{tpm:07.04.2019_2'} 
	\sigma \ll (|B||C|)^{1/3} \sqrt{|A| |D|} |D|^{1/6k} (|B|^2 \E^{+} (C)^{k} \E^{+} (B)^{k-1})^{1/6k}
	%(\E^{+} (B))^{1/3} \,,
	\end{equation} 
	and hence as before 
	\[
	\sigma \ll \sqrt{|A| |D|} |B||C| \cdot \left( \frac{\E^{+} (C)^{k-1} \E^{+} (B)^{k-2}}{|B|^{4k-6}|C|^{4k-4}}\right)^{1/(8k-6)}
	%(\E^{+} (B) |B|^{-4})^{\frac{k-1}{4k-3}} 
	= \sqrt{|A| |D|} |B||C| \rho(B,C) 
	%	\le \sqrt{|A| |D|} |B||C|  (\E^{+} (B) |B|^{-4})^{1/5} 
	%\,.
	\]
	as required. 
	To obtain \eqref{f:A_abs_Re_Z2} we use the same calculations as in Proposition \ref{p:A_abs_Re}.
	This completes the proof.
	$\hfill\Box$
\end{proof}

%\begin{remark}
%	Of course, instead of taking $B\subset \Z$ it is enough to have the condition that there are $O(1)$ elements of $B$ in any segment of length two, say
%	(it easily follows from the ping--pong lemma).
%\end{remark}

\bigskip 

As we have seen the proof of Theorem \ref{t:A_abs_Re_Z} gives us an analogue  of Lemma \ref{l:Lemma 27}, which we formulate in $\Z$ and in  $\F_p$.  
Write $a^{+} := \max\{ a,1\}$.

\begin{lemma}
	Let 
	%$N$ be  a positive integer, 
	$B,C \subseteq \Z$, 
	$\la \neq 0$ be a real number.
	%$|\la| \le |C|/4$ and  
	Put 
	% $G=G_\la (2\cdot[N],2\cdot[N])$. 
	%$G=G(2\cdot[N])$. 
	$G=G_\la (B,C)$. 
	Then for an arbitrary integer $s$ the following holds
	\[
	\T_{2k} (G) \le  	(8|\la|^{+})^{4k} |C|^{3k} |B|^{3k-1}  \,. 
	\]
	Now, let  $B,C \subseteq [N] \subset \F_p$ and $\la \in \F^*_p$,
	$|\la|\le N^2$. 
	%, $|\lambda| \le N/4$.
	Then 
	\begin{equation}\label{cond:Lemma 27_new}
	%	(1+4N^2)^s \le \frac{p}{2} \,,
	\T_{2k} (G) \le (8|\la|)^{4k} \left( \frac{2(2N^2)^{2k}}{p} + 1\right)^4 |B|^{3k-1} |C|^{3k} \,.
	%	N^{3s-1} \,. 
	\end{equation} 
	%provided $|\lambda| \le N^2$. 
	\label{l:Lemma 27_new}	
\end{lemma}
\begin{proof}
	Take $m = 2\lceil |\la| \rceil \le 4|\la|^{+}$ and split $B$ onto odd/even numbers and $C$ onto congruence classes modulo $m$.
	Using Corollary \ref{c:non-abelian_norm} and calculations in \eqref{f:eq_T}, \eqref{f:I_free}, we obtain
	\[
	\T_{2k} (G) \le (2m)^{4k} |B|^2  (\E^{+} (C))^{k} (\E^{+} (B))^{k-1} \le (8|\la|^{+})^{4k} |B|^2  (\E^{+} (C))^{k} (\E^{+} (B))^{k-1} 
	\le 
	(8|\la|^{+})^{4k} |C|^{3k} |B|^{3k-1} %\,. 
	\] 
	as required. 
	
	%	We need to prove bound \eqref{cond:Lemma 27_new} only.
	%Indeed 
	Now let us obtain bound \eqref{cond:Lemma 27_new} and again we split $B,C$ modulo two and $m$, correspondingly,  but before let us remark that
	by the definition of the set   $G$
	%=G_\la (2\cdot[N],2\cdot[N])$ 
	the operator $l^2 (\R^2)$--norm of any element of $g\in G$ is 
	\[
	\| g\| := \sup_{\|x\|_2 =1} \|gx\|_2 \le  \sqrt{1+|a|^2 + |b|^2 + (|ab| +|\lambda|)^2}  \le \sqrt{1+ 2 N^2 + (N^2+N^2)^2} \le 2N^2 \,.
	\]
	Hence $\| g_1 \dots g_{2k} \| \le (2N^2)^{2k}$ and $g_1 \dots g_s \equiv g'_1 \dots g'_{2k} \pmod p$ implies that 
	\begin{equation}\label{f:Garaev_g}
	g_1\dots g_{2k} = g'_1 \dots g'_{2k} + s p \,,
	\end{equation}
	where for a matrix $s$ one has $p \|s\| \in [-(2N^2)^{2k}, (2N^2)^{2k}]$	(see similar arguments in \cite{Margulis}, \cite{Gamburd}, \cite{CG}).
	Clearly, there are at most $(2(2N^2)^{2k} p^{-1} + 1)^4$ of such matrices $s$. 
	Fixing $s$ and $g'_1, \dots, g'_{2k} \in G$ in \eqref{f:Garaev_g}, we need to solve this equation in $g_1,\dots, g_{2k}$.
	Thanks to \eqref{f:I_free} there are at most $|B|^{k-1} |C|^{k}$ choices for   $g_1,\dots, g_{2k}$. 
	This completes the proof.
	$\hfill\Box$
\end{proof}

\begin{remark}
	Notice that there is a  universal way to estimate the energy $\T_{2k} (G_\la (B,C))$ in any field, namely, by Lemma \ref{l:T_2^k} we always have
	$\T^2_{2k} (G_\la (B,C)) \le \T^{+}_{4k} (B)\T^{+}_{4k} (C)$. 
	Again it connects the problem about incidences for hyperbolas with ordinary additive energies of sets.  
\end{remark}

The same arguments as in the proof of Theorem \ref{t:B_progr}, Theorem  \ref{t:A_abs_Re_Z} and Lemma \ref{l:Lemma 27_new} (or see the proof of \cite[Theorem 24]{NG_S}) 	 
%(and the last remark) 
give us an  analogues result for {\it subsets} of $[N]$ and for our two--parametric family of transformations from $\SL_2 (\F_p)$.

\begin{theorem}
	Let $\la \in \F_p^*$, $A,D \subseteq \F_p$ be any sets, %$|A|=|D|$ 
	and $B, C \subseteq [N]$, $|B||C| \ge N^\eps$, 
	%, $C\subseteq [N]$, 
	%$m:= \min\{N,M\} 
	$N \le p^{\tau}$, $0<\tau < 1/8$. 
	Suppose that $|D| \le p^{1-\d}$, $\d = C_1^{-1/\eps}$ and $|\la| \le (|B||C|)^{1/8}/8$.
	Then  
	\begin{equation}\label{f:B_progr_new}
	|\{ (a+b) (c+d) = \lambda ~:~ a\in A,\, b \in B,\, c\in C,\, d\in D \}| 
	\ll
	\sqrt{|A| |D|} |B||C| \cdot (|B||C|)^{- C_2 \d/\eps} \,. 
	\end{equation}
	Here $C_1, C_2 >1$ are absolute constants.
	\label{t:B_progr_new}
\end{theorem}
{\par \noindent \hbox{\rm S\,k\,e\,t\,c\,h\, o\,f\, t\,h\,e\, p\,r\,o\,o\,f.}~}
Let $\sigma$ be the left--hand side of \eqref{f:B_progr_new} and $G=G_\la (-B,C)$. 
We take $m$ such that $N^{4m} \sim p$ and thus by Lemma \ref{l:Lemma 27_new} one has $\T_{2m} (G) \ll (8 |\la|)^{4m} |G|^{3m}$. 
Considering $\nu (g) = r_{(GG^{-1})^{2m}} (g)$, we have by estimate \eqref{f:r_infty} that $\| \nu \|_\infty \ll (8 |\la|)^{4m} |G|^{3m}$
and similar for the  intersection of $\nu$ with any proper subgroup of $\SL_2 (\F_p)$, see \cite{BG}, \cite{NG_S}.  
By Lemma \ref{l:T_2^k} one has
\[
\sigma^{4m} \le |A|^{2m} |D|^{2m-1} \cdot \sum_g r_{(GG^{-1})^{2m}} (g) \sum_x D (x) D (g x) = 
|A|^{2m} |D|^{2m-1} \cdot \sum_g \nu (g) \sum_x D (x) D (g x)\,.
\]
Put $K = |G|^{m/2}$.
Thanks to our condition $|\la| \le (|B||C|)^{1/8}/8$, we have $\| \nu \|_\infty \ll (8 |\la|)^{4m} |G|^{3m} \le |G|^{4m} / K$ 
and similar for the  intersection of $\nu$ with any proper subgroup of $\SL_2 (\F_p)$.
Then by general expansion result in $\SL_2 (\F_p)$, see \cite[Theorem 9]{NG_S} or just formula of Theorem \ref{t:B_any}, we get
\[
\sigma^{4m} \ll \frac{|A|^{2m} |D|^{2m+1} |G|^{4m}}{p} + |A|^{2m} |D|^{2m} |G|^{4m} p^{-\eta} \,,
\]	
where $\eta \ll 2^{-k}$ and $k \ll \log p/\log K \sim \log N/\log |G| \le \eps^{-1}$.
The first term in the last formula is negligible because  our assumption $|D| \le p^{1-\d}$. 
Thus, our saving is $p^{\eta/4m} \gg |G|^{C_2 \d/\eps}$.  
This completes the proof.
$\hfill\Box$

\bigskip

%Actually, one can probably relax the condition on $N$ slightly and 
We write Theorem \ref{t:B_progr_new} 
% this way to compare it with Theorem \ref{t:B_progr}.
similar to Theorem \ref{t:B_progr} for compare these two results. 
Of course constants $C_1,C_2$ in \eqref{f:B_progr_new} are worse than in \eqref{f:B_progr}.

\section{On bilinear forms of Kloosterman sums}
\label{sec:Kloosterman}

%\bigskip
%$\hfill\Box$

Let $\F$ be a finite field, $\a : \F \to \C$, $\beta : \F \to \C$ be two weights and let 
\[
K(n,m) = \sum_{x \in \F^*} e( nx + mx^{-1}) = K(mn, 1) 
\]
be the Kloosterman sum. 
We are interested in bilinear forms of Kloosterman sums \cite{KMS}--\cite{KMS_gen}, that is, expressions 
\[
S(\a,\beta) = \sum_{n,m} \a(n) \beta (m) K(n,m) \,.
\] 
Using the definition of the Fourier transform \eqref{F:Fourier}, we see that
\begin{equation}\label{f:S_Fourier}
S(\a,\beta) = \sum_x \FF{\a} (x) \FF{\beta} (x^{-1}) \,.
\end{equation}
From the Parseval identity \eqref{F_Par} and the Cauchy--Schwarz inequality, we obtain 
\begin{equation}\label{f:S_basic1}
S(\a,\beta) \le p \|\a\|_2 \|\beta\|_2 
\end{equation}
and applying usual upper bound for Kloosterman sum, as well as the Cauchy--Schwarz inequality again, we get
\begin{equation}\label{f:S_basic2}
S(\a,\beta) \le 2\sqrt{p} \|\a\|_1 \|\beta\|_1 \le  2\sqrt{p} \|\a\|_2 \|\beta\|_2 \sqrt{|\supp \a| |\supp \beta|}\,. 
\end{equation}
Both basic bounds \eqref{f:S_basic1}, \eqref{f:S_basic2} give $p^{3/2}$ for, say,  $\a$ and $\beta$ equal the characteristic function of some sets of sizes $\sqrt{p}$. 
This $p^{3/2}$ estimate is a kind of barrier and our task is to beat it for wide range of functions $\a$, $\beta$.

The next general  result 
%shows 
demonstrates 
that 
%the estimation of 
the quantity 
$S(\a,\beta)$  
%how the estimate
is connected with a sum--product question, namely, with the counting of incidences for some hyperbolas. 
Actually, even simple formula \eqref{f:S_Fourier} shows that this problem has the sum--product flavour.
Indeed, suppose for simplicity that $\a$ is the characteristic function of a progression, then the question about estimation of bilinear sums is equivalent to the problem how 
the inverse of a progression correlates with the set of 
large Fourier coefficients of $\beta$. 
%with the inverse of a progression.
In other words, 
it is a question about 
how additive and multiplicative structure agree.

\begin{theorem}
	Let $\a_1, \a_2, \beta_1, \beta_2  : \F_p \to \C$ be functions, $\eps>0$. 
	% and $A=\supp \a$, $B=\supp \beta$.
	Then either 
	\begin{equation}\label{f:d-est-}
	S(\a_1 \a_2,\beta_1 \beta_2) \lesssim p^{1/2+\eps} \min\{ \| \a_1 \|_2 \| \a_2 \|_2 \| \beta_1 \beta_2 \|_2, \| \beta_1 \|_2 \| \beta_2 \|_2 \| \a_1 \a_2 \|_2 \} 
	\end{equation}
	or
	\begin{equation}\label{f:d-est}
	S(\a_1 \a_2,\beta_1 \beta_2) \lesssim \frac{\| \a_1 \|_2 \| \beta_1 \|_2 \| \a_2 \|_1 \| \beta_2 \|_1}{p} + \| \a_1 \|_2 \| \beta_1 \|_2 \| \a_2\|_W \| \beta_2\|_W p^{1-\d}  \,,
	\end{equation}
	where $\d (\eps)>0$ depends on $\eps$ only. 
	%is a positive constant. 
	In particular, if 
	%$\a_2$, $\beta_2$ are the characteristic functions of some arithmetic progressions, then
	$\| \a_2\|_W, \| \beta_2\|_W \lesssim 1$ and if $\| \a_2\|^2_2$ or $\| \beta_2\|^2_2$ is at most $p^{1-c}$, $c>0$, then 
	% or 
	\begin{equation}\label{f:d-est'}
	S(\a_1 \a_2,\beta_1 \beta_2) \lesssim \| \a_1 \|_2 \| \beta_1 \|_2  p^{1-\d}  \,,
	\end{equation}
	where $\d (c) >0$ is a positive constant.
	Here the sign $\lesssim$ depends on $\log (\|\FF{\a}_1 \|_\infty \|\FF{\a}_2\|_\infty \|\FF{\beta}_1\|_\infty \|\FF{\beta}_2\|_\infty)$. 
	\label{t:Kloosterman}
\end{theorem}
\begin{proof}
	%Since $\a(x) = \a(x) A(x)$, $\beta(x) = \beta(x) B(x)$, it follows that via 
	By \eqref{f:S_Fourier} and \eqref{f:inverse}, applied for the convolution, we have
	\[
	S(\a_1 \a_2,\beta_1 \beta_2) = |\F|^{-2} \sum_x (\FF{\a}_1 * \FF{\a}_2)(x) (\FF{\beta}_1 * \FF{\beta}_2) (x^{-1}) \,.
	\]
	Let $A_j = \{ x ~:~ 2^{j-1} < |\FF{\a}_1 (x)| \le 2^{j} \}$, $A'_j = \{ x ~:~ 2^{j-1} < |\FF{\a}_2 (x)| \le 2^{j} \}$,
	$B_j = \{ x ~:~ 2^{j-1} < |\FF{\beta}_1 (x)| \le 2^{j} \}$, $B'_j = \{ x ~:~ 2^{j-1} < |\FF{\beta}_2 (x)| \le 2^{j} \}$.
	The number of such sets is  at most $L = 2\log (\|\FF{\a}_1 \|_\infty \|\FF{\a}_2\|_\infty \|\FF{\beta}_1\|_\infty \|\FF{\beta}_2\|_\infty)$ 
	and our sing $\lesssim$ below depends on this quantity.  
	By the pigeon--hole principle there are $j_1,j_2,j_3,j_4$ and sets $A_{j_1}$, $A'_{j_2}$, $B_{j_3}$, $B'_{j_4}$, which we denote as $A$, $A'$, $B$, $B'$
	and numbers $\D=2^{j_1}$, $\D' = 2^{j_2}$, $\rho=2^{j_3}$, $\rho' = 2^{j_4}$ such that
	\[
	S(\a_1 \a_2,\beta_1 \beta_2) \lesssim  \D \D' |\F|^{-1} \sum_x (A * A') (x) |\FF{\beta_1 \beta_2} (x^{-1})| \lesssim \D \D' \rho \rho'  |\F|^{-2} \sum_x (A * A') (x) (B * B') (x^{-1}) \,.
	\]
	If
	$|A'|$ (or $|B'|$) is
	at most $|\F|^\eps$, then  by \eqref{f:Wiener}, the Parseval and the previous formulae, we have 
	\[
	S^2(\a_1 \a_2,\beta_1 \beta_2) \lesssim (\D \D' |\F|^{-1} )^2 |\F| \|\beta_1 \beta_2 \|_2^2 |A| |A'|^2 
	\le
	|\F|^{1+\eps} \|\beta_1 \beta_2 \|_2^2 \| \a_1\|_2^2 \| \a_2 \|_2^2  
	\]
	%\[
	%	S(\a_1 \a_2,\beta_1 \beta_2) \lesssim \D \D'  |\F|^{-1+c} \| \beta_1 \beta_2\|_W
	%		\le 	
	%				\| \a_1 \|_1 \| \beta_1 \|_W \| \beta_2 \|_W   |\F|^{-1+2c} 
	%			\le
	%			\| \a_1 \|_2 \| \beta_1 \|_2 \| \beta_2 \|_W   |\F|^{2c} 
	%			\,.
	%\]
	and thus we obtain \eqref{f:d-est-}. 
	%it is better, than \eqref{f:d-est} for sufficiently small $c$ (we take $c=1/4$ for simplicity). 
	Now we use the fact that $\F$ equals $\F_p$ to apply our incidences results for hyperbolas. 
	Using 
	%the right--hand side of 
	bound \eqref{f:B_any} from  Theorem \ref{t:B_any}, as well as formula \eqref{f:Wiener}, we get 
	\[
	S(\a_1 \a_2,\beta_1 \beta_2) \lesssim \frac{\D \D' \rho \rho' |A| |B|}{p^3} + \| \a_1 \|_2 \| \beta_1 \|_2 \D' \rho' |A'| |B'| p^{-1-\d} 
	\ll
	\]
	\[ 
	\ll
	\frac{\| \a_1 \|_W \| \beta_1 \|_W \| \a_2 \|_1 \| \beta_2 \|_1}{p} +
	\| \a_1 \|_2 \| \beta_1 \|_2 \| \a_2\|_W \| \beta_2\|_W p^{1-\d}  
	\ll
	\]
	\[
	\ll
	\frac{\| \a_1 \|_2 \| \beta_1 \|_2 \| \a_2 \|_1 \| \beta_2 \|_1}{p} +
	\| \a_1 \|_2 \| \beta_1 \|_2 \| \a_2\|_W \| \beta_2\|_W p^{1-\d}  
	\]
	as required.

	Finally, inequality \eqref{f:d-est'} follows from \eqref{f:d-est-}, \eqref{f:d-est} by the inverse formula \eqref{f:inverse}, which gives 
	\[
	\| \beta_1 \beta_2 \|_2 \le \| \beta_2 \|_\infty \| \beta_1 \|_2 \le \| \beta_2 \|_W \| \beta_1 \|_2 \lesssim  \| \beta_1 \|_2 
	%\,.
	\]
	%	from Lemma \ref{l:P_norms} (for progressions) which says that $\| \a_2\|_W, \| \beta_2\|_W \lesssim 1$. 
	%	(and the same is for Bohr sets, say).
	and hence bound \eqref{f:d-est-} is negligible. 
	The first term in \eqref{f:d-est} is less than the second one, again because  \eqref{f:Wiener} (which gives $\| \a_2 \|_1$, $\| \beta_2 \|_1 \lesssim p$) and our assumption that 
	$\| \a_2\|^2_2$ or $\| \beta_2\|^2_2$ is at most $p^{1-c}$.  
	This completes the proof.
	$\hfill\Box$ 
\end{proof}

\bigskip

%Bound \eqref{f:d-est_intr} of Theorem \ref{t:Kloosterman_intr} follows from \eqref{f:d-est'} because if 

Once again the main advantage of our result is its generality. 
For example, one can easily 
%replace 
consider more general sets than arithmetic progressions in \eqref{f:d-est'}, say,  Bohr sets  of bounded dimension \cite{TV}. 
%It follows from the proof that $5/8$ can be replaced to $1/2+\eps$ and then $\d$ in  = \d(\eps)$.

It is easy to check that the last result is better than \cite[Theorem 7]{Shp_Sato}, as well as \cite[Theorem 1.17(2)]{FKM}. 
Thus, using relatively simple methods from Additive Combinatorics we break $p^{3/2}$ 
%the square root 
barrier in this  problem.

Now we obtain a result on bilinear Kloosterman sums in  a specific situation when the supports of the weights belong to arithmetical progressions. 
%Of course one can replace  the set $[N]+t_1$ below to rather arbitrary function with small Wiener norm in the spirit of Theorem \ref{t:Kloosterman}. 
%For example, in  \eqref{f:Kloosterman_NM_1} one can replace the main term $N^{19/48} M^{7/48} p^{23/24}$ to 
%$\|\FF{\a}\|_{L^{4/3}} N^{7/48} M^{7/48} p^{23/24}$ for the weight $\a$, $\supp \a \subseteq [N]+t_1$  and similar for the main term in \eqref{f:Kloosterman_NM_2}. 

\begin{theorem}
	Let 
	$\a : \F_p \to \C$, 
	$\beta : \F_p \to \C$  be a function, 
	$\supp \a \subseteq [N] + t_1$, 
	$\supp \beta \subseteq [M] +t_2$ and $t_1,t_2 \in \F_p$ be some shifts.  
	Then
	%	\begin{equation}\label{f:Kloosterman_NM}
	%		S(\a,\beta) \lesssim p^{1/6}  \|\FF{\a} \|_{4/3}  \| \beta \|_2 (MN)^{1/6} + (\E^{+} (\a))^{1/4} p^{3/4} + p \| \a \|_W \| \beta \|_W \,,
	%	\end{equation}
	%	and
	\begin{equation}\label{f:Kloosterman_NM_1}
	S(\a,\beta) \lesssim \| \beta \|_2 \left(\|\FF{\a}\|_{L^{4/3}} N^{7/48} M^{7/48} p^{23/24}   
	+  (\|\a\|_2 \| \a\|_1)^{1/2} p^{3/4} + \| \a \|_W  p \right) 
	\,,
	\end{equation}
	and if 
	%$M^2 <pN$, 
	$M^2 N^2 \|\FF{\a}\|^{12}_{L^{4/3}} < p \| \a\|_2^{12}$,  
	then  
	%and $t_1=t_2=0$, then 
	\begin{equation}\label{f:Kloosterman_NM_2}
	S(\a,\beta) \lesssim \| \beta \|_2 \left(\|\FF{\a}\|^{6/7}_{L^{4/3}} \|\a \|^{1/7}_2 N^{1/7} M^{1/7} p^{13/14} 
	+  (\|\a\|_2 \| \a\|_1)^{1/2} p^{3/4} + \| \FF{\a} \|_{L^{4/3}} p^{13/12} \right) \,.
	\end{equation}
	%	Finally, if $N=M \le p^{1/21}$ and $t_1=t_2=0$, then 
	%\begin{equation}\label{f:Kloosterman_NM_3}
	%S(\a,\beta) \lesssim \sqrt{p} \| \a\|_2 \| \beta \|_2 N^{63/64}  \,.
	%\end{equation}
	Here the sign $\lesssim$ depends on $\log (MN \|\FF{\a}\|_\infty \|\FF{\beta}\|_\infty)$. 
	\label{t:Kloosterman_NM}
\end{theorem}
\begin{proof}
	We can suppose that $N,M$ and $p$ are sufficiently large because otherwise the result is trivial. 
	Let us begin with \eqref{f:Kloosterman_NM_1}.
	Let $B$ and $C$ be the characteristic functions of the arithmetic progressions $[N]+t_1$ and $[M]+t_2$, respectively. 
	%	Also, let $\beta' (x) = \beta (-x)$. 
	Then for any weights $\a \subseteq B$, $\beta \subseteq C$ we can write $\a = \a B$ and $\beta = \beta C$. 
	After that we repeat the arguments of the proof of Theorem \ref{t:Kloosterman}. 
	%	but before notice that thanks to  \eqref{f:S_Fourier} and Lemma \ref{l:P_norms} the following holds 
	%\begin{equation}\label{tmp:13.04_-1}
	%	S(\a, \beta) \le p \| B \|_W \sum_x |\FF{B} (x)| |\FF{\beta} (x^{-1})|^2 \lesssim p \sum_x |\FF{\beta} (x^{-1})|^2 |\FF{B} (x)| = p S_* (\a,\beta) \,.
	%\end{equation}
	%Further, 
	Namely, 
	splitting the level sets of the functions 
	%$C-C,\beta * \beta', \a,B$, 
	$\FF{\a}, \FF{B}, \FF{C}, \FF{\beta}$, 
	we obtain sets $A$, $B',C',D$ and positive numbers $\D_1,\D_2,\D_3, \D_4$ such that 
	\begin{equation}\label{tmp:13.04_1}
	S (\a, \beta) (p^2 \D_1\D_2\D_3 \D_4)^{-1} \lesssim  
	\frac{|A||B'||C'||D|}{p} + |A|^{1/4} |B'||C'| |D|^{1/2} + 
	\end{equation}
	\begin{equation}\label{tmp:13.04_1-}
	+ |A|^{3/4} |D|^{1/2} (|B'||C'|)^{41/48} \,. 
	%		\left( 1+ \left(\frac{|B'||C'|}{p} \right)^{1/12} \right)  \,.	
	\end{equation}
	Here we have used Corollary \ref{c:hyp_incidences} 
	and again in \eqref{tmp:13.04_1} and below our sing $\lesssim$ depends on  
	%$L = \log (MN\|\FF{\beta}\|_\infty )$.
	$L = \log (MN \|\FF{\a}\|_\infty \|\FF{\beta}\|_\infty)$.
	We will show later that the first two terms in \eqref{tmp:13.04_1} give the last two terms in 
	%\eqref{f:Kloosterman_NM}, 
	\eqref{f:Kloosterman_NM_1} and now let us consider the third term in \eqref{tmp:13.04_1-}.  
	%are negligible. 
	From Parseval identity \eqref{F_Par}, we have 
	\begin{equation}\label{tmp:13.04_2'-}
	(\D^2_4 |D|)^{1/2} \le \sqrt{p} \| \beta \|_2 \,.
	\end{equation}
	%	and, similarly, 
	%\begin{equation}\label{tmp:13.04_2}
	%	(\D^{4/3}_1 |A|)^{3/4} \le \|\FF{\a} \|_{4/3} 
	%	%\le p^{3/4} \| \a \|^{4/3}_2 \,.
	%\end{equation}
	In view of Lemma \ref{l:P_norms}, we get 
	\begin{equation}\label{tmp:13.04_2'}
	(\D^{48/41}_2 |B'|)^{41/48} \ll p^{41/48} N^{7/48}
	\end{equation}
	and the same for $A$ and $C'$. 
	Thus, we obtain 
	%(one can check that the term $\left(\frac{|B'||C'|}{p}\right)^{1/12}$ in \eqref{tmp:13.04_1-} is negligible)
	\[
	S(\a,\beta) \lesssim 
	p^{23/24}   
	\| \beta \|_2 \|\FF{\a}\|_{L^{4/3}} (NM)^{7/48} 
	\]
	as required.
	It remains to show that two terms in \eqref{tmp:13.04_1} give the last two terms in  \eqref{f:Kloosterman_NM_1}.
	%	(the arguments for \eqref{f:Kloosterman_NM} are similar).   
	%are negligible. 
	In view of Lemma \ref{l:P_norms} and inequality  \eqref{f:Wiener} the first one gives 
	\[
	(\D_1\D_2\D_3 \D_4) \frac{|A||B'||C'||D|}{p^3} \lesssim  p \| \a \|_W \| \beta \|_W \ll p \| \a \|_W \| \beta \|_2
	\]
	and 	by the same lemma and the Parseval identity, as well as \eqref{f:energy_Fourier}, \eqref{tmp:13.04_2'-}, \eqref{tmp:13.04_2'}, we have for the second  term 
	\[
	(\D_1\D_2\D_3 \D_4) \frac{|A|^{1/4}|B'||C'||D|^{1/2}}{p^2} \lesssim p^{3/4} \E^{+} (\a)^{1/4} \| \beta \|_2 
	\lesssim  
	p^{3/4} (\|\a\|_2 \| \a\|_1)^{1/2} \| \beta \|_2 \,.
	\]

	Now let us prove \eqref{f:Kloosterman_NM_2}. 
	Let $Q_1$, $Q_2$ be symmetric arithmetic progressions with steps equal one such that 
	\begin{equation}\label{cond:Q_1,Q_2}
	|Q_1| N \ll p \,, \quad \quad |Q_2| M \ll p \,.
	\end{equation}
	Let $Z \subseteq \F_p$ be any set and write $\FF{\a}_{Z} (r):= \FF{\a} (r) Z(r)$ and similar for $\FF{\beta}$ 
	(in this part of the proof one can put, simply, $Z=\F_p$).
	% but we need a more general case later).   
	If $t_1 = 0$, then we have
	\[
	\| \FF{\a}_Z (r) - |Q_1|^{-1} (\FF{\a}_Z * Q_1) (r) \|_2^2 
	=
	|Q_1|^{-2} \sum_{r\in Z} \left| \sum_{s\in Q_1} (\FF{\a}_Z (r) - \FF{\a}_Z (r+s)) \right|^2
	=
	\]
	\[
	= |Q_1|^{-2} \sum_{r\in Z} \left| \sum_{x = 1}^N \a(x) e(-rx) \sum_{s\in Q_1} (1-e(-sx)) \right|^2 
	\le
	\]
	\[
	\le
	|Q_1|^{-2} \sum_{r\in \F_p} \left| \sum_{x = 1}^N \a(x) e(rx) \sum_{s\in Q_1} (1-e(sx)) \right|^2 
	=
	|Q_1|^{-2} 
	p \sum_{x = 1}^N |\a(x)|^2 \left| \sum_{s\in Q_1} (1-e(sx)) \right|^2 
	\]
	and because $Q_1$ is a symmetric set, as well as conditions \eqref{cond:Q_1,Q_2}, we obtain
	\[
	\| \FF{\a}_Z (r) - |Q_1|^{-1} (\FF{\a}_Z * Q_1) (r) \|_2^2 \ll p |Q_1|^{-2}  \sum_{x = 1}^N |\a(x)|^2 \left|\sum_{s\in Q_1} \frac{|s|^2 |x|^2}{p^2} \right|^2 
	\ll
	\| \a\|_2^2 \cdot \frac{N^4 Q^4_1}{p^3} \,.
	\]
	The same holds for $\beta$
	and below we will 
	%denote by  
	write 
	$\FF{\a} (r) = \FF{\a_0} (r) e(-t_1 r)$, where $\supp \a_0 \subseteq [N]$ and similar for $\beta$. 
	Clearly, $\| \a \|_{L^q} = \| \a_0 \|_{L^q}$, $\|\a\|_q = \|\a_0 \|_q$ and $\| \beta \|_{L^q} = \| \beta_0 \|_{L^q}$, $\|\beta\|_q = \|\beta_0 \|_q$ for any $q$. 
	Hence by \eqref{f:S_Fourier}, the Cauchy--Schwarz inequality and formula \eqref{F_Par}, we get 
	\[
	S(\a,\beta) = \sum_r \FF{\a_0} (r) e(-rt_1) 
	\FF{\beta} (r^{-1})
	= 
	|Q_1|^{-1} \sum_r (\FF{\a_0} * Q_1) (r) e(-rt_1) 
	\FF{\beta} (r^{-1}) + \|\a\|_2 \| \beta \|_2 \cdot  O\left( \frac{N^2 Q^2_1}{p} \right)
	%	=
	\]
	\[
	= 
	|Q_1|^{-1} \sum_r (\FF{\a_0} * Q_1) (r^{-1}) \FF{\beta_0} (r) e(-t_2 r^{} - t_1 r^{-1})
	+ \|\a\|_2 \| \beta \|_2 \cdot  O\left( \frac{N^2 Q^2_1}{p} \right)
	=
	\]
	\begin{equation}\label{f:S_Q_1,Q_2}
	%S(\a,\beta) 
	= (|Q_1| |Q_2|)^{-1} \sum_r (\FF{\a_0} * Q_1) (r) (\FF{\beta_0} * Q_2) (r^{-1})  e(-t_2 r^{-1} - t_1 r^{}) + \|\a\|_2 \| \beta \|_2 \cdot  O\left( \frac{N^2 Q^2_1}{p} + \frac{M^2 Q^2_2}{p} \right) \,.
	\end{equation}
	Now our task is to estimate the first sum in \eqref{f:S_Q_1,Q_2}, which we denote as $\sigma$. 
	As before splitting the level sets of the functions 
	$\FF{\a}_0$, $\FF{\beta}_0$,	we obtain sets $A$, $D$  and numbers $\D_1,\D_2$ such that 
	\begin{equation}\label{f:S_Q_1,Q_2+}
	\sigma \lesssim \D_1 \D_2 (|Q_1| |Q_2|)^{-1} \sum_r (A * Q_1) (r) (D * Q_2) (r^{-1}) \,.
	\end{equation}
	The trick with $\FF{\a}_{Z} (r)$, $\FF{\beta}_{Z} (r)$ allows us to choose $Q_1$, $Q_2$ not depending on the sets $A,D$ 
	%and we will need it at the third part of the proof. 
	(but here, actually,  we do not need in this additional information). 
	Applying Corollary \ref{c:hyp_incidences_progr}, we have 
	\begin{equation}\label{tmp:17.04_1}
	\sigma (\D_1 \D_2)^{-1} \lesssim \frac{|A||D|}{p} + 
	|A|^{1/4} |D|^{1/2} + |A|^{3/4} |D|^{1/2} (|Q_1||Q_2|)^{-1/6} \left( 1+ \left(\frac{|Q_1||Q_2|}{p} \right)^{1/12} \right)  \,.
	\end{equation}
	Again two first terms in the last formula do not exceed  two last terms in \eqref{f:Kloosterman_NM_2}. 
	Additionally, we consider the term $\left(\frac{|Q_1||Q_2|}{p} \right)^{1/12}$ later.
	As before, one has
	\[
	\D_1 \D_2 |A|^{3/4} |D|^{1/2} (|Q_1||Q_2|)^{-1/6} \ll p^{5/4} \| \FF{\a} \|_{L^{4/3}} \| \beta \|_2 (|Q_1||Q_2|)^{-1/6} \,.
	\]
	Using the last formula and recalling \eqref{f:S_Q_1,Q_2}, we see that the optimal choice of the parameters $Q_1,Q_2$ is 
	$NQ_1 = M Q_2$ and hence  $Q_1 = p^{27/28} M^{1/14} N^{-13/14} (\| \FF{\a} \|_{L^{4/3}} \| \a \|^{-1}_2 )^{3/7}$, 
	$Q_2 = p^{27/28} M^{-13/14} N^{1/14} (\| \FF{\a} \|_{L^{4/3}} \| \a \|^{-1}_2 )^{3/7}$.
	The 
	%choice 
	assumption 
	%$M^2 < pN$ 
	$M^2 N^2 \|\FF{\a}\|^{12}_{L^{4/3}} < p \| \a\|^{12}_2$ 
	guarantees that  conditions \eqref{cond:Q_1,Q_2} hold. 
	Thus 
	\[
	S(\a,\beta) \lesssim \| \beta \|_2 N^{1/7} M^{1/7}  \|\FF{\a}\|^{6/7}_{L^{4/3}} \|\a \|^{1/7}_2 p^{13/14} \,.
	\]
	Finally, let us consider the situation when the term  $\left(\frac{|Q_1||Q_2|}{p} \right)^{1/12}$ in \eqref{tmp:17.04_1} dominates.
	In this case $|Q_1| |Q_2| > p$ and
	\[
	\D_1 \D_2 |A|^{3/4} |D|^{1/2} (|Q_1||Q_2|)^{-1/6} \left(\frac{|Q_1||Q_2|}{p} \right)^{1/12}
	\le 
	p^{5/4}  \| \FF{\a} \|_{L^{4/3}} \| \beta \|_2 p^{-1/6} = p^{13/12}  \| \FF{\a} \|_{L^{4/3}} \| \beta \|_2
	\]
	as required. 	
	%This completes the proof.
	$\hfill\Box$
\end{proof}

\bigskip

%\begin{remark}
%	One can use Theorem \ref{t:B_progr} in the proof of the third part of  Theorem \ref{t:Kloosterman_NM} but it gives the same as Thereom \ref{t:Kloosterman}. 
%	Also, one can see from the proof of the third bound \eqref{f:Kloosterman_NM_3} that taking $N$ sufficiently small relatively to $p$, that is, $1 \ll N\le p^\eps$, we obtain the saving $\| \a \|_2 \| \beta \|_2 \sqrt{p} N^{2-1/8 + \d(\eps)}$, where $\d(\eps) \to 0$ as $\eps \to 0$.  
%\end{remark}

It was proved in \cite[Theorem 6.1]{BFKMM} that under some mild assumptions on $N$, $M$ and the case of the initial interval $[N]$ one has 
\[
S([N],\beta) \ll (\| \beta \|_1 \| \beta \|_2)^{1/2} p^{3/4+o(1)}  M^{1/12} N^{7/12} \,.
\] 
%For the critical case $M=N=p^{1/2}$ 
%%this bound 
%the saving here 
%coincides with \eqref{f:Kloosterman_NM_1}. 
Our bound \eqref{f:Kloosterman_NM_1} is better 
%for large $M$ 
(let $\beta (x) = [M](x)$ for simplicity)
%, namely, 
in the case when $p^{10} \ll M^{9} N^9$ and $M^2\gg N^{}$, $M^{4} N^7 \gg p^3$.  
Obviously, 
%our  
more precise estimate \eqref{f:Kloosterman_NM_2}  is even better.

\bigskip

\noindent{I.D.~Shkredov\\
	Steklov Mathematical Institute,\\
	ul. Gubkina, 8, Moscow, Russia, 119991}
\\
and
\\
IITP RAS,  \\
Bolshoy Karetny per. 19, Moscow, Russia, 127994\\
and 
\\
MIPT, \\ 
Institutskii per. 9, Dolgoprudnii, Russia, 141701\\
{\tt ilya.shkredov@gmail.com}

\end{document}